\documentclass[11pt]{article}
\usepackage{multicol}
\usepackage{amssymb,amsmath,amsthm,bm}
\usepackage[colorlinks=true, urlcolor=blue,linkcolor=blue, citecolor=blue]{hyperref}
\usepackage{mathrsfs,verbatim}
\usepackage{caption,cleveref}
\usepackage{graphicx, float}

\usepackage{color, mathtools, tikz, xcolor}
\usepackage{enumerate}
\usepackage{tikz}
\usetikzlibrary{shapes.geometric}
\usepackage{pdfpages}
\usepackage{circuitikz}
\usepackage{subcaption}
\usepackage{enumerate,enumitem}

\usetikzlibrary{calc}

\newcommand{\E}{\mathbb{E}}
\newcommand{\Pro}{\mathbb{P}}

\usepackage{bbm}
\usepackage[mathscr]{euscript}
\usepackage{appendix}
\usepackage{geometry}
\newcommand{\no}{\noindent}

\numberwithin{equation}{section}

\begin{document}

	\newcommand{\bea}{\begin{eqnarray}}
		\newcommand{\ena}{\end{eqnarray}}
	\newcommand{\beas}{\begin{eqnarray*}}
		\newcommand{\enas}{\end{eqnarray*}}
	\newcommand{\beq}{\begin{align}}
		\newcommand{\enq}{\end{align}}
	\def\qed{\hfill \mbox{\rule{0.5em}{0.5em}}}
	\newcommand{\bbox}{\hfill $\Box$}
	\newcommand{\ignore}[1]{}
	\newcommand{\ignorex}[1]{#1}
	\newcommand{\wtilde}[1]{\widetilde{#1}}
	\newcommand{\mq}[1]{\mbox{#1}\quad}
	\newcommand{\bs}[1]{\boldsymbol{#1}}
	\newcommand{\qmq}[1]{\quad\mbox{#1}\quad}
	\newcommand{\qm}[1]{\quad\mbox{#1}}
	\newcommand{\nn}{\nonumber}
	\newcommand{\Bvert}{\left\vert\vphantom{\frac{1}{1}}\right.}
	\newcommand{\To}{\rightarrow}
	\newcommand{\supp}{\mbox{supp}}
	\newcommand{\law}{{\cal L}}
	\newcommand{\Z}{\mathbb{Z}}
	\newcommand{\mc}{\mathcal}
	\newcommand{\mbf}{\mathbf}
	\newcommand{\tbf}{\textbf}
	\newcommand{\lp}{\left(}
	\newcommand{\limm}{\lim\limits}
	\newcommand{\limminf}{\liminf\limits}
	\newcommand{\limmsup}{\limsup\limits}
	\newcommand{\rp}{\right)}
	\newcommand{\mbb}{\mathbb}
	\newcommand{\rainf}{\rightarrow \infty}
	%Erd\"{o}s-R\'{e}nyi
	\newtheorem{theorem}{Theorem}[section]
	\newtheorem{problem}[theorem]{Problem}
	\newtheorem{exercise}[theorem]{Exercise}
	\newtheorem{corollary}[theorem]{Corollary}
	\newtheorem{conjecture}[theorem]{Conjecture}
	\newtheorem{claim}[theorem]{Claim}
	\newtheorem{proposition}[theorem]{Proposition}
	\newtheorem{lemma}[theorem]{Lemma}
	\newtheorem{definition}[theorem]{Definition}
	\newtheorem{example}[theorem]{Example}
	\newtheorem{remark}[theorem]{Remark}
	\newtheorem{solution}[theorem]{Solution}
	\newtheorem{case}{Case}[theorem]
	\newtheorem{condition}[theorem]{Condition}
	\newtheorem{assumption}[theorem]{Assumption}
	\newtheorem{question}[theorem]{Question}
	\newtheorem{note}[theorem]{Note}
	\newtheorem{notes}[theorem]{Notes}	\newtheorem{observation}[theorem]{Observation}
	\newtheorem{readingex}[theorem]{Reading exercise}
	\frenchspacing

	\newenvironment{proofclaim}[1][Proof of claim]{\begin{proof}[#1]\renewcommand*{\qedsymbol}{\(\blacksquare\)}}{\end{proof}}
	
	\newcommand{\indicate}{\mathbbm{1}}
	
	% PACKAGES FOR NEW FIGURES START
	
	\begin{comment}
		\usepackage{amssymb}
		\usepackage{graphicx,bm}
		\usepackage{enumerate}
		\usepackage{amsthm,amscd}
		\usepackage[all]{xy}
		
		\usepackage{graphicx} % Required for inserting images
		\usepackage{tikz} 
	\end{comment}
	
	\usetikzlibrary{positioning}
	\usetikzlibrary{decorations.markings}
	\usetikzlibrary{shapes.geometric}
	\usetikzlibrary{patterns}
	\usetikzlibrary{calc, babel}

	\tikzstyle{vertex}=[circle,fill=red!24,minimum size=20pt,inner sep=0pt]
	\tikzstyle{selected vertex} = [vertex, fill=red!24]
	\tikzstyle{edge} = [draw,thick,-]
	\tikzstyle{weight} = [font=\small]
	\tikzstyle{selected edge} = [draw,line width=5pt,-,red!50]
	\tikzstyle{ignored edge} = [draw,line width=5pt,-,black!20]
	
	\pgfdeclarelayer{edgelayer}
	\pgfdeclarelayer{nodelayer}
	\pgfsetlayers{edgelayer,nodelayer,main}
	
	\tikzstyle{none}=[inner sep=0pt]
	\definecolor{RED}{rgb}{1.000,0.000,0.000}
	\definecolor{ORANGE}{rgb}{1.000,0.700,0.000}
	\definecolor{BLUE}{rgb}{0.000,0.000,1.000}
	\definecolor{GREEN}{rgb}{0.000,0.700,0.300}
	\definecolor{GRAY}{rgb}{0.5,0.5,0.5}
	\definecolor{SHADEDGRAY}{rgb}{0.9,0.9,0.9}
	\definecolor{WHITE}{rgb}{1.000,1.000,1.000}
	\definecolor{BLACK}{rgb}{0.000,0.000,0.000}

	\tikzstyle{whitecircle}=[circle,fill=WHITE,draw=BLACK]
	\tikzstyle{whiterectangle}=[rectangle,fill=WHITE,draw=BLACK]
	\tikzstyle{blackcircle}=[circle,fill=BLACK,draw=BLACK]
	\tikzstyle{grayrectangle}=[rectangle,fill=GRAY,draw=BLACK]
	
	\tikzstyle{thinarrow}=[-latex,draw=BLACK,line width=0.700]
	\tikzstyle{thickedge}=[-,draw=BLACK,line width=2.500]
	
	\tikzstyle{REDARROW}=[-latex,draw=RED,line width=2.00]
	\tikzstyle{ORANGEARROW}=[-latex,draw=ORANGE,line width=2.00]
	\tikzstyle{BLUEARROW}=[-latex,draw=BLUE,line width=2.00]
	\tikzstyle{GREENARROW}=[-latex,draw=GREEN,line width=2.00]
	\tikzstyle{GRAYARROW}=[-latex,draw=GRAY,line width=2.00]
	\tikzstyle{BLACKARROW}=[-latex,draw=BLACK,line width=2.00]
	\tikzstyle{THINBLACKARROW}=[-latex,draw=BLACK,line width=1.00]

	% PACKAGEs FOR NEW FIGURES END

	\newcommand{\DELETE}[1]{\textcolor{yellow}{#1}}
	
	\newcommand{\ADD}[1]{\textcolor{red}{#1}}
	
	\newcommand{\MAY}[1]{\textcolor{blue}{Akif: #1}}

	\tikzstyle{level 1}=[level distance=2.75cm, sibling distance=5.65cm]
	\tikzstyle{level 2}=[level distance=3cm, sibling distance=2.75cm]
	\tikzstyle{level 3}=[level distance=3.9cm, sibling distance=1.5cm]
	
	% Define styles for bags and leafs
	\tikzstyle{bag} = [text width=10em, text centered] 
	\tikzstyle{end} = [circle, minimum width=3pt,fill, inner sep=0pt]

	\title{Path Decompositions of Oriented Graphs}
	\author{Viresh Patel\footnote{School of Mathematical Sciences, Queen Mary University of London, United Kingdom,  Email: viresh.patel@qmul.ac.uk. V.~Patel was partially supported by the Netherlands Organisation for Scientific Research (NWO)
			through the Gravitation project NETWORKS (024.002.003).} \hspace{0.2in} Mehmet Akif Yıldız\footnote{Centrum Wiskunde \& Informatica, Amsterdam, The Netherlands. Email: m.akif.yildiz@cwi.nl. M.A.~Yıldız was supported by a Marie Skłodowska-Curie Action from the EC (COFUND grant no. 945045) and by the NWO Gravitation project NETWORKS (grant no. 024.002.003). Part of this work was carried out during the tenure of an `ERCIM Alain Bensoussan' Fellowship Programme. }   } \vspace{0.2in}
	
	\maketitle
	\begin{abstract}
		%We consider the problem of decomposing the edges of a digraph into as few paths as possible. It is not too hard to see that a natural lower bound for the number of paths in any path decomposition is $\frac{1}{2}\sum_{v\in V(D)}|d^+(v)-d^-(v)|$. Any digraph that achieves this bound is called consistent. We say a graph $G$ is \textit{strongly consistent} if every orientation of $G$ is consistent. As a generalization of Kelly's conjecture on Hamilton decompositions of regular tournaments, Alspach, Mason, and Pullman \cite{Directed-Path-Tournaments} conjectured in 1976 that every tournament of even order is consistent. This was recently verified for large tournaments by Girão, Granet, Kühn, Lo, and Osthus \cite{Tournaments}, i.e.~$K_{2n}$ is strongly consistent for large $n$. A more general conjecture of Pullman \cite{CUBIC} states that every regular graph of odd degree is strongly consistent. In this paper, we prove that a random regular graph of odd degree is strongly consistent with high probability. In fact, we find a deterministic class of strongly consistent graphs for which a random regular graph of odd degree belongs to this class with high probability. Along the way, we verify the conjecture for graphs with no short cycles.	
		We consider the problem of decomposing the edges of a digraph into as few paths as possible. A natural lower bound for the number of paths in any path decomposition of a digraph $D$ is $\frac{1}{2}\sum_{v\in V(D)}|d^+(v)-d^-(v)|$; any digraph that achieves this bound is called consistent. 
		%As a generalization of Kelly's conjecture on Hamilton decompositions of regular tournaments, 
		Alspach, Mason, and Pullman \cite{Directed-Path-Tournaments} conjectured in 1976 that every tournament of even order is consistent and this was recently verified for large tournaments by Girão, Granet, Kühn, Lo, and Osthus \cite{Tournaments}. 
		A more general conjecture of Pullman \cite{CUBIC} states that for  odd $d$, every orientation of a $d$-regular graph is consistent. We prove that the conjecture holds for random $d$-regular graphs with high probability i.e. for fixed odd $d$ and as $n \to \infty$ the conjecture holds for almost all $d$-regular graphs. 
		%In fact, we find a deterministic family of graphs whose every orientation is consistent such that a random regular graph of odd degree belongs to this family with high probability. 
		Along the way, we verify Pullman's conjecture for graphs whose girth is sufficiently large (as a function of the degree).
		
		%A digraph $D$ is said to be consistent if it can be decomposed into $\mathrm{ex}(D)$ many paths, where $\mathrm{ex}(D):=\frac{1}{2}\sum_{v\in V(D)}|d^+(v)-d^-(v)|$ is a natural lower bound for any path decomposition. We say a graph $G$ is \textit{strongly consistent} if every orientation of $G$ is consistent. Let $d$ be an odd natural number. As a generalization of Kelly's conjecture, Alspach, Mason, and Pullman conjectured in 1976 that $K_{d+1}$ is strongly consistent. This is recently verified for large $d$ by Girão, Granet, Kühn, Lo, and Osthus. In this paper, we prove that a random $d$-regular graph is strongly consistent with high probability. In fact, we find a deterministic class of strongly consistent graphs for which a random $d$-regular graph belongs to this class with high probability. We also prove that $d$-regular graphs with girth at least $200d^2$ and cubic graphs are strongly consistent.
		\bigskip
		
		
	\end{abstract}
	
	\section{Introduction}\label{sec:intro}
	
	\no Decomposing an object into smaller parts that share some common structure is a frequent theme in mathematics. Many classical and fundamental problems in graph theory can be phrased as \textit{decomposition} problems, i.e.\ problems of partitioning the edge set of a graph into parts with certain properties, where the goal is usually to minimize the number of parts.  \\

	%For paths, Gallai (see \cite{Gallai}) conjectured that every $n$-vertex connected graph can be decomposed into at most $\lceil n/2\rceil$ paths; in fact this conjecture was inspired by a famous result of Lov{\'a}sz who showed the conjecture holds if one allows both paths and cycles. It has been shown for a few graph classes e.g. for planar graphs by Blanché, Bonamy and Bonichon \cite{Gallai-planar}. The best bound for general graphs is $\lfloor 2n/3 \rfloor$ paths  due to Dean and Kouider \cite{Gallai-2n-over-3-first} and independently Yan \cite{Gallai-2n-over-3-second}.
	
	\no
	A famous result of Lov{\'a}sz~\cite{Gallai} says that every $n$-vertex graph can be decomposed into at most $\lfloor n/2 \rfloor$ paths and cycles. This result was inspired by a conjecture of Gallai (see~\cite{Gallai}) which says that every connected $n$-vertex graph can be decomposed into at most $\lceil n/2 \rceil$ paths, and this is sharp for complete graphs with an odd number of vertices. The conjecture remains open, where currently the best general bound is $\lfloor 2n/3 \rfloor$
	due to Dean and Kouider \cite{Gallai-2n-over-3-first} and independently Yan \cite{Gallai-2n-over-3-second}. The conjecture is also known to hold for several graph classes (see e.g.~\cite{Gallai-planar} and references therein). \\
	
	\no
	In this paper, we are interested in the problem of decomposing \textit{directed} graphs into as few \textit{paths} as possible.
	The obvious analogue of Gallai's conjecture for directed graphs is false because there are directed graphs that need many more than $\lceil n/2 \rceil$ paths to decompose them. For example $\lfloor n^2/4 \rfloor$ paths are needed to decompose the balanced complete bipartite digraph that sends all its edges from one vertex class to the other since each edge must be a path; in fact $\lfloor n^2/4 \rfloor$ paths are enough to decompose any digraph on at least four vertices as shown by O'Brien~\cite{Path-Decomposition-Upper-Bound} (building on work of Alspach and Pullman~\cite{Directed-Path-Decomposition}). However, whereas Gallai's conjecture is close to sharp for many graphs (e.g. graphs where all degrees are odd require at least $n/2$ paths to decompose them), $\lfloor n^2/4 \rfloor$ is far from sharp for most digraphs (e.g.\ a bound of $n^2/4$ is meaningless for sparse digraphs). Therefore one is interested in more subtle bounds for digraphs. \\

	\no In this direction, Alspach and Pullman \cite{Directed-Path-Decomposition} gave a natural lower bound on the number of paths needed to decompose a digraph in terms of the degrees of the vertices.
	%They also proved that any $n$-vertex oriented graph can be decomposed into at most $n^2/4$ paths, with equality holding for transitive tournaments. O'Brien \cite{Path-Decomposition-Upper-Bound} generalized this upper bound for any digraph on at least four vertices.\\
	Let $\mathrm{pn}(D)$ be the minimum number of paths required to decompose the edges of a digraph $D$. As observed in \cite{Directed-Path-Decomposition}, $\mathrm{pn}(D)$ can be bounded below by the the excess of $D$, which is defined to be
	\begin{align*}
		\mathrm{ex}(D):= \dfrac{1}{2}\sum_{v\in V(D)} |d^+(v) - d^-(v)|.
	\end{align*}  
	Indeed, we have $\mathrm{pn}(D) \geq \mathrm{ex}(D)$ because, for any path decomposition of $D$, the number of paths that start (resp.\ end) at a vertex $v$ is at least $\max(d^+(v) - d^-(v), 0)$ (resp.\ $\max(d^-(v) - d^+(v), 0)$), and summing this over all vertices counts each path exactly twice. 
	Any digraph satisfying $\mathrm{pn}(D)=\mathrm{ex}(D)$ is said to be \textit{consistent}. We say an undirected graph $G$ is \textit{strongly consistent} if every orientation of $G$ is consistent.\\
	
	\no 
	%It is easy to show that acyclic digraphs are consistent; see \cite{Directed-Path-Decomposition} (also see \Cref{prop:acyclic-consistent}). 
	It is clear that not every digraph is consistent (e.g.~a directed cycle or any Eulerian digraph has excess zero), and in fact it is NP-complete to determine whether a digraph is consistent; see \cite{NP-complete-paper} or \cite{NP-complete-thesis}. 
	%However, a recent result \cite{RANDOM-DIRECTED-CONSISTENT} shows that a (dense) random digraph is consistent with high probability. 
	A conjecture of Alspach, Mason, and Pullman from 1976 states that every tournament of even order is consistent; in our terminology the conjecture  can be stated as follows. 	
	\begin{conjecture}[\cite{Directed-Path-Tournaments}]\label{conjecture:Tournaments}
		For each odd $d$, the complete graph on $d+1$ vertices is strongly consistent.
	\end{conjecture}	
	
	\no Girão, Granet, Kühn, Lo, and Osthus \cite{Tournaments} recently resolved Conjecture~\ref{conjecture:Tournaments} for large $d$, building on earlier work of the second author with Lo, Skokan, and Talbot~\cite{Tournaments-Viresh}. In fact, Conjecture~\ref{conjecture:Tournaments} is a generalisation of Kelly’s conjecture about decomposing regular tournaments into Hamilton cycles, which was solved by Kühn and Osthus~\cite{Kelly-Proof} for large tournaments using their robust expanders technique.
	%which states that every regular tournament (every vertex has the same in- and out-degree) on $2n+1$ vertices can be decomposed into $n$ Hamilton cycles. Indeed, the removal a vertex and all of its incident edges from a regular tournament on $2n+1$ vertices results in an orientation of a complete graph on $2n$ vertices where every vertex has excess one. Then, Conjecture~\ref{conjecture:Tournaments} would imply the existence of a path decomposition with exactly $n$ paths, each of them can be easily completed into a Hamilton cycle in the original tournament using the vertex previously removed. 
	The following conjecture attributed to Pullman (see \cite{CUBIC}) considerably generalizes Conjecture~\ref{conjecture:Tournaments} from complete graphs to all regular graphs.
	
	\begin{conjecture}\label{conjecture:Regular}
		For odd $d$, every $d$-regular graph is strongly consistent.
	\end{conjecture}
	
	\no We note that Conjecture~\ref{conjecture:Regular} is immediate for $d=1$, and was verified for $d=3$ by Reid and Wayland \cite{CUBIC}, but is wide open beyond that. (Note that Eulerian digraphs show that Conjecture~\ref{conjecture:Regular} cannot hold for even $d$.)
	Our main result is to verify Conjecture~\ref{conjecture:Regular} for random regular graphs.
	
	\begin{theorem}\label{thm:RANDOM-REGULAR}
		For each fixed odd $d$, a uniformly random $n$-vertex, $d$-regular graph is strongly consistent with probability tending to $1$ as $n \to \infty$.    
	\end{theorem}
	
	\no Secondly, we verify Conjecture~\ref{conjecture:Regular} for graphs with no short cycles. Write $g(G)$ for the girth of a graph $G$.
	
	\begin{theorem}\label{corollary:main}
		For odd $d$, every $d$-regular graph $G$ with $g(G)\geq 200d^2$ is strongly consistent.
	\end{theorem}
	
	\no In fact, our results are more general. 
	%We find a deterministic class of strongly consistent graphs (see \Cref{thm:discrete-orientation}) for which a random $d$-regular graph belongs to this class with high probability (see \Cref{prop:RANDOM-REGULAR}). Additionally, 
	We prove that for any graph (not necessarily regular), every orientation without short cycles is consistent, provided that $d^+(v) \not= d^-(v)$ for every vertex $v$ (see \Cref{theorem:LARGE-GIRTH}). Then \Cref{corollary:main} follows as an immediate corollary. Furthermore, we show that it is possible to allow some vertices satisfying $d^+(v) = d^-(v)$, as long as they are sufficiently far apart from each other (see \Cref{theorem:KEY}). \\
	
	%	The key property about random regular graphs used in the proof of \Cref{thm:RANDOM-REGULAR} concerns the short cycles. Moreover, we use the fact that $d$ being odd only for ensuring every vertex has nonzero excess in every orientation. Indeed, if a digraph has no short cycles, and has nonzero excess at every vertex, we prove that it is consistent.
	
	\no While the robust expanders technique was used heavily in the previous related works \cite{Tournaments, Kelly-Proof, Tournaments-Viresh} mentioned above, it is not applicable in our setting of sparse graphs. 
	%The techniques of the paper lie within the general paradigm of absorption, which has been a very successful idea for decomposition problems. 
	Instead, we develop a new absorption idea tailored to the sparse setting of Pullman's Conjecture.

	\subsection*{Further related work}
	\no Path and cycle decomposition problems are among the most widely studied problems in extremal graph theory. Such problems go as far back as the 19th century with Walecki \cite{Walecki} showing that the edges of a complete graph of odd order can be decomposed into Hamilton cycles. There are many beautiful problems and results in the area, and we mention only a few of them first for graphs and then for digraphs. \\
	
	\no 
	We have already mentioned Gallai's conjecture about decomposing graphs into paths. For cycles, Erdős and Gallai~\cite{Erdos-Gallai} conjectured that any $n$-vertex graph can be decomposed into $O(n)$ cycles and edges. There has been extensive work on this conjecture~(see e.g.~references in \cite{Towards-Erdos-Gallai}) including recent progress of Conlon, Fox and Sudakov~\cite{Conlon-Fox-Sudakov} giving a bound of $O(n \log\log n)$ cycles and currently the best bound of $O(n\log^* n)$ cycles due to Bucić and Montgomery \cite{Towards-Erdos-Gallai} (here $\log^* n$ is the iterated logarithm function). Another important problem, due to Akiyama, Exoo and Harary~\cite{Arboricity}, is the \textit{Linear Arboricity Conjecture}, which asks for a decomposition of a $d$-regular graph into at most $\frac{d+1}{2}$ linear forests (a linear forest is a vertex disjoint union of paths). Again, this conjecture has received considerable attention; we mention here only that it is known to hold asymptotically as shown by Alon \cite{Alon-Arboricity}, and currently the best  {bounds are} $\frac{d}{2} + 3\sqrt{d} \log^4 d$  due to Lang and Postle \cite{Arboricity-Lang-Postle}, and $\frac{d}{2}+O(\log n)$ due to  Christoph, Draganić, Girão, Hurley, Michel, and Müyesser \cite{NEW-LINEAR-ARBORICITY}, where the comparison between the two best bounds depends on whether $d=\Omega(\log^2 n)$ or not. In a related direction, one can ask for a path decomposition in which path lengths are specified; for example Kotzig~\cite{KOTZIG} asked which $d$-regular graphs (for odd $d$) can be decomposed into paths of length $d$. Recently Montgomery, Müyesser, Pokrovskiy, and Sudakov~\cite{ALP} have made progress on this problem by showing all $d$-regular graphs can be almost decomposed into paths of length roughly $d$.
	\\
	
	\no In the directed setting, we have already mentioned the long-standing conjecture of Kelly (now settled for large tournaments), which eventually led to Conjectures~\ref{conjecture:Tournaments} and~\ref{conjecture:Regular} a few years later. 
	A directed analogue of the   {Erdős-Gallai} Conjecture says that any directed Eulerian graph\footnote{Note that the  {Erdős-Gallai} Conjecture is equivalent to the conjecture that every Eulerian graph can be decomposed into $O(n)$ cycles since every graph can be written as an edge-disjoint union of an Eulerian graph and a forest.  {Here we use a slightly general definition for Eulerian (directed) graphs; we say a graph (resp.~directed graph) $G$ is Eulerian if $d(v)$ is even (resp.~$d^+(v)=d^-(v)$) for all $v\in V(G)$.}} can be decomposed into $O(n)$ cycles (in fact an exact number of cycles was conjectured by Jackson~\cite{JACKSON-DECOMPOSITION} and later updated 
	by Dean~\cite{DEAN-DECOMPOSITION}.) The best bound here is $O(n \log n)$ \cite{EULERIAN-DIRECTED-LOG-DELTA}.
	A directed analogue of the linear arboricity conjecture was first mentioned by Nakayama and Peroch\'e \cite{Directed-Arboricity-First} who conjectured that every $d$-regular digraph can be decomposed into at most $d+1$ linear diforests. He, Li, Bai, and Sun~\cite{Directed-Arboricity-Complete} observed that the conjecture does not hold for complete digraphs on three or five vertices, but conjectured that these are the only exceptions. The conjecture is known to be true for digraphs with large girth due to  Alon \cite{Directed-Arboricity-Alon}, and for complete digraphs due to He, Li, Bai, and Sun~\cite{Directed-Arboricity-Complete}, but not much is known beyond this. There does not seem to be any formulation of Kotzig's question in the directed setting, but it is worth noting that while Kotzig's question for (even) cliques asks to decompose the clique into Hamilton paths, the obvious analogue in the oriented setting would ask to decompose even tournaments that are almost regular into Hamilton paths; this is a much harder problem which is in fact equivalent to Kelly's conjecture.
	
	\subsection*{The basic idea}	
	
	%The starting point for our proofs is the basic idea introduced in \cite{Tournaments-Viresh}, which we describe informally here. 
	Given an orientation $D$ of some graph (with some further properties), we wish to decompose $D$ into $\mathrm{ex}(D)$ paths. By iteratively removing cycles from $D$ until it is no longer possible, we can write $D$ as the edge-disjoint union $D = E \cup A$, where $E$ is an Eulerian digraph and $A$ is an acyclic digraph.  In fact, $\mathrm{ex}(A) = \mathrm{ex}(D)$ (and $\mathrm{ex}(E) = 0$) because each removal of a cycle preserves $|d^+(v) - d^-(v)|$ for every vertex $v$. Furthermore, it is easy to show that any acyclic digraph is consistent simply by iteratively removing maximal paths and observing that the excess of the digraph decreases by one with each removal. Putting all this together, we see that $D$ can be decomposed into an Eulerian digraph $E$ together with a family $\mathcal{P}$ of $\mathrm{ex}(D) = \mathrm{ex}(A)$ paths. While we have the correct number of paths, not all of the edges (namely those in $E$) have been covered by these paths. We know $E$ can be decomposed into a set of cycles $\mathcal{C}$ and the challenge now is to absorb each cycle in $\mathcal{C}$ into some of the paths in $\mathcal{P}$. More precisely, for each cycle $C \in \mathcal{C}$, we try to find paths $P_1, \ldots, P_k \in \mathcal{P}$ so that $C \cup P_1 \cup \cdots \cup P_k$ can be rewritten as the edge-disjoint union of $k$ new paths $P_1' \cup \cdots \cup P_k'$; in this way the total number of paths remains $\mathrm{ex}(D)$.
	By repeating this, we eventually use up all the edges in $E$ and end up with the correct number (i.e.\ $\mathrm{ex}(D)$) of paths.\\
	
	\no In order to make the absorption step work, we must choose $\mathcal{P}$ and 
	$\mathcal{C}$ carefully. In 
	fact $\mathcal{P}$ (as described above) is not mentioned explicitly in our proofs and is constructed somewhat 
	indirectly. On the other hand, $\mathcal{C}$ is given explicitly. One might expect that it is desirable to 
	have as few cycles as possible in $\mathcal{C}$ so that less absorption is needed; however, surprisingly, we actually maximize the number of cycles in $\mathcal{C}$ because it 
	gives us some structural control over the cycles.

	\section{Preliminaries}\label{sec:preliminaries}
	Throughout the paper, we use standard graph theory notation and terminology. For $k\in\mathbb{N}$, we sometimes denote the set $\{1,2,\ldots,k\}$ by $[k]$. Throughout, a \textit{digraph} $D$ is a directed graph without loops where, for any distinct vertices $u$ and $v$ in $D$, there are at most two edges between $u$ and $v$, at most one in each direction. A digraph $D$ is an \textit{oriented} graph if it has at most one edge between any two distinct vertices, i.e.\ $D$ can be obtained by orienting the edges of an undirected graph. For a (di)graph $G$, we denote the vertex set and the edge set of $G$ by $V(G)$ and $E(G)$, respectively. For $a,b\in V(G)$, we write $ab$ for the (directed) edge from $a$ to $b$. For an oriented graph $D$, we denote the \textit{underlying undirected graph of $D$} by $D^\diamond$. We say an oriented graph $D$ is a \textit{tournament} if $D^\diamond$ is a complete graph, that is $ab\in E(D^\diamond)$ for every distinct $a,b\in V(D^\diamond)$. We write $H\subseteq G$ to mean $H$ is a \textit{sub(di)graph} of the (di)graph $G$, that is $V(H)\subseteq V(G)$ and $E(H)\subseteq E(G)$. The \textit{(di)graph induced} by $F\subseteq E(G)$, denoted by $G[F]$, is the sub(di)graph of $G$ with vertex set consisting of those vertices incident to edges in $F$ and with edge set $F$. Similarly, the \textit{(di)graph induced} by $S\subseteq V(G)$, denoted by $G[S]$, is the sub(di)graph of $G$ with the vertex set $S$ and edge set consisting of those edges which have both endpoints in $S$. For simplicity, we sometimes write $F$  instead of $G[F]$. A subset $S\subseteq V(G)$ is called an \textit{independent set} if $G[S]$ has no edges. For a collection $\mathcal{Q}$ of (di)graphs, we write $V(\mathcal{Q})=\bigcup_{Q\in \mathcal{Q}}V(Q)$ and $E(\mathcal{Q})=\bigcup_{Q\in \mathcal{Q}}E(Q)$. For a collection $\mathcal{Q}$ of sub(di)graphs of $G$, we write $G-\mathcal{Q}$ for the (di)graph obtained from $G$ by deleting all the edges in $E(\mathcal{Q})$. We say a graph $G$ is \textit{bipartite with bipartition $(A,B)$} if $V(G)=A\cup B$ and $E(G)\subseteq \{ab:a\in A, b\in B\}$, and we indicate $G$ is such a bipartite graph by writing $G=G(A,B)$.\\
	
	\no \textbf{Neighborhoods and degrees:} For a graph $G$ and a vertex $v\in V(G)$, $N_G(v)$ denotes the set of \textit{neighbours} of $v$, that is the set of vertices $w\in V(G)$  for which $wv\in E(G)$. We denote the \textit{degree} of a vertex $v$ by $d_G(v)$, that is $d_G(v)=|N_G(v)|$. Similarly, for a digraph $D$ and $v\in V(D)$, let $N_{D}^{+}(v)$ and $N_{D}^{-}(v)$ denote the set of \textit{outneighbours} and \textit{inneighbours} of $v$, that is the set of vertices $w\in V(D)$ for which $vw\in E(D)$ and $wv\in E(D)$, respectively. We denote the \textit{outdegree} and \textit{indegree} of $v$ by $d_{D}^{+}(v)$ and $d_{D}^{-}(v)$, respectively, that is, $d_{D}^{+}(v)=|N_{D}^+(v)|$ and $d_{D}^{-}(v)=|N_{D}^-(v)|$. For $S\subseteq V(D)$, we write $N_{D}^+(S):=\bigcup_{v\in S}N^+(v)$ and $N_{D}^-(S):=\bigcup_{v\in S}N^-(v)$. We often omit subscripts when they are clear from the context. For directed graphs, we define $d(v):=d^+(v)+d^-(v)$ as the \emph{total degree} of $v$. The \textit{minimum degree} of a graph $G$ is $\delta(G):=\min\{d(v):v\in V(G)\}$. Similarly, the \textit{maximum degree} of $G$ is $\Delta(G):=\max\{d(v):v\in V(G)\}$. For a digraph $D$, the \textit{minimum semi-degree} of $D$ is $\delta^0(D):=\min\{ \min\{d^+(v),d^-(v)\}:v\in V(D)\}$, and, the \textit{maximum semi-degree} of $D$ is defined as $\Delta^0(D):=\max\{ \max\{d^+(v),d^-(v)\}:v\in V(D)\}$. We say a graph (resp. digraph) $G$ is \textit{$k$-regular} if $d(v)=k$ (resp. $d^{+}(v)=d^{-}(v)=k$) for all vertices $v\in V(G)$. \\ 
	%Throughout the paper, we work on (di)graphs with no \textit{isolated} vertices, i.e.~the (total) degree of every vertex is nonzero.
	
	\no \textbf{Paths and cycles:} A \textit{directed path} $P$ in a digraph $D$ is a subdigraph of $D$ such that $V(P)=\{v_1,\ldots,v_k\}$ for some $k\in\mathbb{N}$ and $E(P)=\{v_1v_2,v_2v_3,\ldots,v_{k-1}v_k\}$. We denote such a directed path by its vertices in order, i.e.~we write $P=v_1v_2\ldots v_k$. The \textit{length} of a path $P$ is denoted by $|P|$, that is the number of edges in $P$. A \textit{directed cycle} in $D$ is the same except that it also includes the edge $v_kv_1$. We sometimes omit the word ``directed" and only say path or cycle if it is obvious from the context. We write $\ell(P)$ and $r(P)$ for the starting and ending point of a path $P$, that is $\ell(P)=v_1$ and $r(P)=v_k$. We sometimes treat an edge $ab$ as a path with $\ell(P)=a$ and $r(P)=b$. For a directed path $P=v_1v_2\ldots v_k$ and $1\leq i<j\leq k$, we write $v_iPv_j$ for the directed path $v_iv_{i+1}\ldots v_{j-1}v_j$, that is the unique directed path from $v_i$ to $v_j$ along $P$. Similarly, for a directed cycle $C$ and $x,y\in V(C)$, we write $xCy$ for the unique directed path from $x$ to $y$ along $C$. For two directed paths $P$ and $P'$ with $\ell(P)=a$, $r(P)=\ell(P')=b$, $r(P')=c$, {and $V(P)\cap V(P')=\{b\}$,} we write $aPbP'c$ for the directed path that is obtained by \textit{gluing} paths $P$ and $P'$ at the vertex $b$. For a (di)graph $G$ and $x,y\in V(G)$, we denote the \textit{distance} from $x$ to $y$ by $\mathrm{dist}_G(x,y)$, that is the length of a shortest path $P$ in $G$ with $\ell(P)=x$ and $r(P)=y$. For a (di)graph $G$, the length of the shortest (directed) cycle is called \textit{girth}, and denoted by $g(G)$. We say a (di)graph $G$ is \textit{acyclic} if it has no (directed) cycle, and in that case, we set $g(G)$ to be infinity. A collection of edge-disjoint (directed) paths (resp. cycles) $\mathcal{Q}$ in a (di)graph $G$ is called a \textit{path family} (resp. \textit{cycle family}) in $G$. The definitions of \textit{path} and \textit{cycle} extend to graphs in the obvious ways. \\
	
	%A \textit{matching} $M$ in a (di)graph $G$ is a set of edges such that every vertex of $G$ is incident to at most one edge in $M$. In other words, it is a path family consisting of vertex-disjoint paths of length one. We say a matching $M$ \textit{saturates} $S\subseteq V(G)$ if every vertex of $S$ is incident to exactly one edge in $M$. The definitions of \textit{path} and \textit{cycle} extend to graphs in the obvious ways.\\
	
	\no The following are non-standard definitions but will be used frequently in the paper.\\
	
	\no \textbf{Tangent and incident paths:}  We say a path $P$ is \textit{incident} to a vertex $v$ if $v\in \{\ell(P),r(P)\}$. We say $P$ is \textit{incident} to a cycle $C$ (at $v$) if $E(C)\cap E(P)=\emptyset$ and there exists $v\in V(C)$ such that $P$ is incident to $v$. We say $P$ is \textit{tangent} to $C$ if either $V(P)\cap V(C)=\{\ell(P)\}$ or $V(P)\cap V(C)=\{r(P)\}$.  \\
	%If $M=\{a_ib_i:i\in [m]\}$ is a matching in a digraph, we write $V^+(M):=\{a_i:i\in [m]\}$ and $V^-(M):=\{b_i:i\in [m]\}$.
	
	\no We now discuss some basic results we require. We postpone the preliminaries on random graphs to the start of \Cref{sec:random-regular}. The first proposition is a standard application of Hall's theorem.
	
	\begin{proposition}\label{prop:HALL-CONSEQUENCE}
		Let $G=G(A,B)$ be a bipartite graph. Assume there exists a positive integer $t$ such that $d(a)\geq t\cdot d(b)$  for all $a\in A$ and $b\in B$. Then, there exists a collection $\{R_a: a\in A\}$ of pairwise disjoint sets such that $R_a\subseteq N(a)$ and $|R_a|=t$ for all $a\in A$.
	\end{proposition}
	\begin{comment}	
		\begin{proof}
			Let $G'=(A',B)$ be the bipartite graph where $A'$ is obtained from $A$ by creating distinct copies $\{a^1,\ldots,a^t\}$ for each $a\in A$ such that $a^ib\in E(G')$ if and only if $ab\in E(G)$ for all $a\in A$, $b\in B$, $i\in [t]$. Note that $d_{G'}(b)=t\cdot d_{G}(b)$ for every $b\in B$, so $d_{G'}(x)\geq d_{G'}(b)$ for all $x\in A'$ and $b\in B$. Hence, by Hall's theorem, $G'$ has a matching $M$ saturating $A'$. The result follows by gluing the edges in $M$ which are incident to $\{a^1,\ldots,a^t\}$ at the vertex $a$ for every $a\in A$.
		\end{proof}
	\end{comment}
	
	\no We will use the following form of Lovász Local Lemma.
	\begin{lemma}[see e.g.~\cite{LocalLemma}]\label{prop:LOCAL-LEMMA}
		Let $\{A_i:i\in [t]\}$ be a set of events such that each event occurs with probability at most $p$ and such that each event is independent of all the other events except for at most $d$ of them. If $epd\leq 1$, then there is a positive probability that none of the events occurs.    
	\end{lemma}
	
	%\no We postpone the preliminaries for random regular graphs into the beginning of \Cref{sec:random-regular}.
	\begin{comment}
		The following provides a bound for the upper tails of a Poisson distributed random variable.
		\begin{proposition}[Bennett's inequality for Poisson random variables, \cite{BENNETT}]\label{prop:BENNETT}
			Let $X$ be a random variable that has Poisson distribution with parameter $\lambda$. Then, $\Pro(X\geq \lambda(1+u))\leq e^{-\lambda h(u)}$ for $u\geq 0$ where $h(u):=(1+u)\log(1+u)-u$.
		\end{proposition}
	\end{comment}

	%	\textbf{Preliminaries for probability:} (jointly) convergence in distribution, Markov's inequality, Bennett's inequality,

	\subsection*{Path decomposition in digraphs}\label{subsec:Path-Decomposition}
	\noindent For a digraph $D$, we say a path family $\mathcal{P}$ is a \textit{path decomposition} for $D$ if $\bigcup_{P\in \mathcal{P}}E(P)=E(D)$. Using this notation, the \textit{path number} of $D$ is given by 
	\begin{align*}
		\mathrm{pn}(D):=\min\{r:\text{there exists a path decomposition }\mathcal{P}\text{ of }D\text{ with }|\mathcal{P}|=r\}.    
	\end{align*}
	For a vertex $v\in V(D)$, recall that the \textit{excess} of $v$ is
	\begin{align*}
		\mathrm{ex}_D(v):=\max\{\mathrm{ex}_D^+(v),\mathrm{ex}_D^-(v)\},
	\end{align*}
	where $\mathrm{ex}_D^+(v)=\max\{d_D^{+}(v)-d_D^{-}(v),0\}$ and $\mathrm{ex}_D^-(v)=\max\{d_D^{-}(v)-d_D^{+}(v),0\}$. The \textit{excess} of $D$ is
	\begin{align*}
		\mathrm{ex}(D):=\sum_{v\in V(D)}\mathrm{ex}_D^+(v)=\sum_{v\in V(D)}\mathrm{ex}_D^-(v)=\dfrac{1}{2}\sum_{v\in V(D)}\mathrm{ex}_D(v).
	\end{align*} We say $v$ is a \textit{plus} (resp. \textit{minus}) vertex if $ex^+(v)>0$ (resp. $ex^-(v)>0$). Let $V^{+}(D)$ and $V^{-}(D)$ be the set of plus and minus vertices in $D$, respectively. Similarly, let $V^0(D):=\{v\in V(D):\mathrm{ex}(v)=0\}$ be the set of \textit{excess-zero} vertices. We say a path $P$ is a \textit{plus-minus path (in $D$)} if $\ell(P)\in V^+(D)$ and $r(P)\in V^-(D)$. For convenience, we state the observation discussed in \Cref{sec:intro} formally, and expand on it. 
	
	\begin{observation}[see~\cite{Directed-Path-Decomposition}]\label{prop:basic}
		Let $D$ be a digraph, and $\mathcal{P}$ be any path decomposition for $D$. Then, for all $v\in V(D)$, we have $|\{P\in \mathcal{P}:\ell(P)=v\}|\geq ex^+(v)$ and $|\{P\in \mathcal{P}:r(P)=v\}|\geq ex^-(v)$, which in particular implies that $|\mathcal{P}|\geq \mathrm{ex}(D)$. Moreover, if $|\mathcal{P}|=\mathrm{ex}(D)$, then we have
		\begin{enumerate}
			\item[(i)] $|\{P\in \mathcal{P}:\ell(P)=v\}|= ex^+(v)$ and $|\{P\in \mathcal{P}:r(P)=v\}|= ex^-(v)$ for all $v\in V(D)$;
			\item[(ii)] Every $P\in \mathcal{P}$ is a plus-minus path in $D$.
		\end{enumerate}
	\end{observation}
	
	\no We say a path family $\mathcal{P}$ is a \textit{perfect path decomposition} for $D$ if $\mathcal{P}$ is a path decomposition for $D$ with $|\mathcal{P}|=\mathrm{ex}(D)$. A digraph that admits a perfect path decomposition is said to be \textit{consistent}. Not all graphs are consistent, e.g.~a directed cycle has excess zero, but it is easy to see that every acyclic digraph is consistent by successively removing maximal paths.
	
	\begin{proposition}[{\cite[Theorem~4]{Directed-Path-Decomposition}}]\label{prop:acyclic-consistent}
		Any acyclic digraph is consistent.    
	\end{proposition}
	\begin{comment}
		\begin{proof}
			We will prove the statement by the induction on the number of edges. For the base case, if a digraph has only one edge, the statement is trivial. Let $D$ be an acyclic digraph, and assume that the claim holds for acyclic digraphs with fewer edges. Take a maximal path $P$ in $D$, let $\ell(P)=x$ and $r(P)=y$. Since $D$ is acyclic, it is easy to check that $N^-(x)=N^+(y)=\emptyset$. This in particular implies that $x\in V^+(D)$ and $y\in V^-(D)$. Therefore, $\mathrm{ex}(D-P)=\mathrm{ex}(D)-1$. Since $D-P$ is consistent by the induction hypothesis, it admits a path decomposition with $\mathrm{ex}(D)-1$ paths, so the result follows.
		\end{proof}
	\end{comment}	
	
	\no We say a path family $\mathcal{P}$ is a \textit{partial path decomposition} for $D$ if $|\{P\in \mathcal{P}:\ell(P)=v\}|\leq ex^+(v)$ and $|\{P\in \mathcal{P}:r(P)=v\}|\leq ex^-(v)$ for all $v\in V(D)$. Note that, if $\mathcal{P}$ is a partial path decomposition, then every $P\in \mathcal{P}$ is a plus-minus path.%For acyclic digraphs, any partial path decomposition can be completed to a perfect path decomposition.
	
	\begin{proposition}\label{prop:PARTIAL-COMPLETION}
		Let $D$ be a digraph, and $\mathcal{Q}$ be a partial path decomposition for $D$. If $D-\mathcal{Q}$ is consistent, then $D$ has a perfect path decomposition $\mathcal{P}$ with $\mathcal{Q}\subseteq \mathcal{P}$. 
		%Let $A$ be an acyclic digraph and $\mathcal{Q}$ be a partial path decomposition for $A$. Then, $A$ has a perfect path decomposition $\mathcal{P}$ with $\mathcal{Q}\subseteq \mathcal{P}$.    
	\end{proposition}
	\begin{proof}
		Let $D':=D-\mathcal{Q}$. 
		Note that for all $v\in V^+(D)$, $\mathrm{ex}^+_{D'}(v)=\mathrm{ex}_D^+(v)-|\{Q\in \mathcal{Q}:\ell(Q)=v\}|\geq 0$ as $\mathcal{Q}$ is a partial path decomposition. Hence, summing over all $v\in V(D')$ gives $\mathrm{ex}(D')=\mathrm{ex}(D)-|\mathcal{Q}|$.
		%Note that $\ell(Q)\in V^+(D)$ and $r(Q)\in V^-(D)$ for all $Q\in \mathcal{Q}$, so we have $\mathrm{ex}(D')=\mathrm{ex}(D)-|\mathcal{Q}|$. 
		If $D'$ has a path decomposition $\mathcal{P}_0$ with exactly $\mathrm{ex}(D')$ many paths, then the result follows by setting $\mathcal{P}=\mathcal{Q}\cup \mathcal{P}_0$.
		%	Let $A':=A-\mathcal{Q}$. Note that $\ell(Q)\in V^+(D)$ and $r(Q)\in V^-(D)$ for all $Q\in \mathcal{Q}$, so we have $\mathrm{ex}(A')=\mathrm{ex}(A)-|\mathcal{Q}|$. Since $A'$ is acyclic, by \Cref{prop:acyclic-consistent}, it has a path decomposition $\mathcal{P}_0$ with exactly $\mathrm{ex}(A')$ many paths. The result follows by setting $\mathcal{P}=\mathcal{Q}\cup \mathcal{P}_0$. 
	\end{proof}
	
	\begin{remark}\label{remark:ACYCLIC-COMPLETION}
		If $A$ is an acyclic digraph, \Cref{prop:acyclic-consistent,prop:PARTIAL-COMPLETION} imply that  any partial path decomposition for $A$ can be completed to a perfect path decomposition. 
	\end{remark}

	\begin{comment}
		For a (di)graph $G$, we say a subset $A\subseteq V(G)$ is \textit{$k$-sparse} if for every distinct $x,y\in A$, the distance between $x$ and $y$ in (the underlying graph of) $G$ is at least $k$. Similarly, we say $A$ is \textit{almost $k$-sparse} if for every $x\in A$, there is at most one $y\in A$, $y\neq x$ whose distance to $y$ in (the underlying graph of) $G$ is less than $k$. 
	\end{comment}
	
	\section{Graphs with no short cycles}\label{sec:main}
	
	In this section, we prove the following theorem, from which \Cref{corollary:main} follows immediately.
	
	\begin{theorem}\label{theorem:LARGE-GIRTH}
		Let $d\in \mathbb{N}$, and $D$ be a digraph with $\Delta^0(D)\leq d$ and $g(D)\geq 200d^2$ such that $\mathrm{ex}_D(v)\neq 0$ for all $v\in V(D)$. Then, $D$ is consistent.
	\end{theorem}
	
	\no In fact, it is possible to allow excess-zero vertices in \Cref{theorem:LARGE-GIRTH} as long as they are far away from each other. For a digraph $D$, we call $S\subseteq V(D)$ \textit{$k$-sparse} in $D$ if for all $u\in S$, there is at most one vertex $v\in S\setminus\{u\}$ satisfying $\mathrm{dist}_{D^\diamond}(u,v)\leq k$. Our general result is the following. 
	%We say a digraph $D$ is \textit{$k$-zero sparse} if $V^0(D)$ is $k$-sparse in $D$. 

	%For a natural number $k$, we say a digraph $D$ is \textit{$k$-zero sparse} if for all $u\in V^0(D)$, there is at most one $v\in V^0(D)$ satisfying $\mathrm{dist}_{D^\diamond}(u,v)\leq k$ other than $u$. Our general result is the following.
	
	\begin{theorem}\label{theorem:KEY}
		Let $d\in \mathbb{N}$, and $D$ be a digraph with $\Delta^0(D)\leq d$ and $g(D)\geq 1000d^2$ such that $V^0(D)$ is $k$-sparse in $D$ for some $k\geq 20d^2+2$. Then, $D$ is consistent.
		%	For natural numbers $d$ and $k$ with $k\geq 6d+2$, every $k$-zero sparse digraph $D$ with $\Delta^0(D)\leq d$ and $g(D)\geq 900d^3$ is consistent.
	\end{theorem}
	
	\no We note that \Cref{theorem:KEY} will be an important part in the proof of   {\Cref{thm:RANDOM-REGULAR}. In~\Cref{sec:random-regular}, we use \Cref{theorem:KEY} in the proof of a stronger result \Cref{thm:discrete-orientation}, from which the proof of \Cref{thm:RANDOM-REGULAR} will be immediate}. We first need a proposition that has a similar flavour to finding an independent transversal.
	
	\begin{proposition}\label{prop:MATCHING}
		Let $d\in \mathbb{N}$, and $G$ be a graph with $\Delta(G)\leq d$. Let $\{X_i:i\in[t]\}$ be a collection of disjoint subsets of $V(G)$ such that $|X_i|\geq 25d$ for each $i\in [t]$. Then $G$ has an independent set $X$ of size $2t$ with $|X\cap X_i|=2$ for all $i\in [t]$. 
	\end{proposition}
	\begin{proof}
		Without loss of generality, we can assume that $|X_i|=25d$ for each $i\in [t]$ (or else we can remove some vertices
		from $X_i$), and that $V(G)=\bigcup_{i\in [t]}V(X_i)$. For each $i\in [t]$, pick two vertices $a_i,b_i\in X_i$ uniformly at random (without replacement) {, where the choices are made independently of one another}. Let $X=\bigcup_{i\in [t]}\{a_i,b_i\}$. For any edge $f\in E(G)$, let $A_f$ be the event that both endpoints of $f$ lie in $X$. Note that $X$ is an independent set if and only if none of the events $\{A_f:f\in E(G)\}$ occur. Hence, it suffices to show that $\mathbb{P}\left(\bigcap_{f\in E(G)}A_f^c\right)>0$. \\
		
		%Every vertex $x\in V(G)$ lies in $X$ with probability $1-\binom{25d-1}{2}/\binom{25d}{2}=2d^{-1}/25$. 
		\no Consider an edge $f=uv\in E(G)$ with $u\in X_i$ and $v\in X_j$ for some $i,j\in [t]$. Note that if $i\neq j$, we have
		\begin{align*}
			\Pro(A_f)=\Pro(u,v\in X)=\left(1-\dfrac{\binom{25d-1}{2}}{\binom{25d}{2}}\right)^2=\dfrac{4}{625d^2}.   
		\end{align*}
		Similarly, if $i=j$, we have
		\begin{align*}
			\Pro(A_f)=\Pro(u,v\in X)=\dfrac{1}{\binom{25d}{2}}=\dfrac{2}{25d(25d-1)}\leq \dfrac{4}{625d^2}.
		\end{align*}
		As a result, we have $\Pro(A_f)\leq 4d^{-2}/625$. Moreover, if $A_{f}$ and $A_{g}$ are dependent for some $g\in E(G)$, then $g$ should be incident to a vertex in $X_i\cup X_j$. This means that the event $A_f$ is independent of all the other events $\{A_g:g\in E(G)\}$ except for at most $2\cdot 25d\cdot d=50d^2$ of them. By \Cref{prop:LOCAL-LEMMA}, the result follows as $e\cdot (4d^{-2}/625)\cdot 50d^2= 8e/25 <1$. 
	\end{proof}

	\no The following will be the key lemma in the proofs of \Cref{theorem:LARGE-GIRTH,theorem:KEY}. We say a cycle family $\mathcal{C}$ in a digraph $D$ is \textit{maximal} if $D-\mathcal{C}$ is acyclic.

	\begin{lemma}\label{prop:KEY}
		Let $d\in \mathbb{N}$, and $D$ be a digraph with $\Delta^0(D)\leq d$. Suppose that there exists a maximal cycle family $\mathcal{C}=\{C_i:i\in [t]\}$ in $D$ and a path family $\{Q_{ij}:i\in [t],\, j\in [100d]\}$ of plus-minus paths in $D-\mathcal{C}$ such that $Q_{ij}$ is tangent to $C_i$ for all $i\in [t]$ and $j\in [100d]$. Then, $D$ is consistent.
	\end{lemma}
	\begin{proof}
		%Let $t:=|\mathcal{C}|$, $r:=|\mathcal{Q}|$, and $\mathcal{A}:=D-\mathcal{C}$. Write $\mathcal{C}=\{C_i:i\in [t]\}$ and $\mathcal{Q}=\{Q_{j}:j\in[r]\}$. Define an undirected bipartite graph $G_1=(\mathcal{C},\mathcal{Q})$ such that for all $i\in [t]$ and $j\in [r]$, we have $C_iQ_j\in E(G_1)$ if and only if $Q_j$ is tangent to $C_i$. It is given that $d_{G_1}(C_i)\geq 200d^2$ for all $i\in [t]$. On the other hand, if $Q_j$ is tangent to $C_i$, by definition we have either $\ell(Q_j)\in V(C_i)$ or $r(Q_j)\in V(C_i)$, so $d_{G_1}(Q_j)\leq 2d$ for all $j\in [r]$ as $\Delta^0(D)\leq d$. Then, by Proposition~\ref{prop:HALL-CONSEQUENCE}, there exists a collection $\{\mathcal{R}_i:i\in [t]\}$ of pairwise disjoint subsets of $[r]$ such that for all $i\in [t]$ and $j\in \mathcal{R}_i$, we have $|\mathcal{R}_i|=100d$ and $Q_j$ is tangent to $C_i$. For notational convenience, we write $Q_{ij}=Q_j$ if $j\in \mathcal{R}_i$. Let $z_{ij}$ be the unique vertex in $V(Q_{ij})\cap V(C_i)$.\\
		
		\no Let $A:=D-\mathcal{C}$, and $z_{ij}$ be the unique vertex in $V(Q_{ij})\cap V(C_i)$ for $i\in [t]$ and $j\in [100d]$. Since $\mathcal{C}$ is a cycle family, we have $\mathrm{ex}_D(v)=\mathrm{ex}_A(v)$ for all $v\in V(D)$, which, in particular, implies that $\mathrm{ex}(D)=\mathrm{ex}(A)$. Note that $A$ is acyclic as $\mathcal{C}$ is maximal. We emphasize that $\{Q_{ij}:i\in [t],\, j\in [100d]\}$ is not necessarily a partial path decomposition for $A$. We claim that for each $i\in [t]$, there exist distinct $a_i,b_i\in [100d]$ such that
		\begin{enumerate}[label={\rm (\alph*)}]
			\item $\bigcup_{i\in [t]}\{Q_{ia_i},Q_{ib_i}\}$ is a partial path decomposition for $A$, \label{item:Q-is-partial-decomposition}
			\item for all $i\in[t]$, {we have $z_{ia_i}\neq z_{ib_i}$, and} either $z_{ia_i},z_{ib_i}\in V^+(D)$ or $z_{ia_i},z_{ib_i}\in V^-(D)$. \label{item:endpoints-same-sign}
		\end{enumerate}
		
		\no {For each $i\in [t]$, let} $\mathcal{Q}^{+}_i=\{Q_{ij}:z_{ij}\in V^{+}(D),\,j\in [100d]\}$ and $\mathcal{Q}^{-}_i=\{Q_{ij}:z_{ij}\in V^{-}(D),\,j\in [100d]\}$. Note that either $|\mathcal{Q}^{+}_i|\geq 50d$ or $|\mathcal{Q}^{-}_i|\geq 50d$. Let $\mathcal{Q}_i$ be the largest one among $\mathcal{Q}^+_i$ and $\mathcal{Q}^-_i$.  {Next}, we define an undirected graph $G$ with vertex set $\mathcal{Q}:=\bigcup_{i\in [t]}\mathcal{Q}_i$ and such that for any distinct $Q,Q'\in \mathcal{Q}$, we have $QQ'\in E(G)$ if and only if either $\ell(Q)=\ell(Q')$ or $r(Q)=r(Q')$. It is easy to see that $\Delta(G)\leq 2d-2$ as $\Delta^0(D)\leq d$. Since $|\mathcal{Q}_i|\geq 50d$ for all $i\in [t]$, by applying Proposition~\ref{prop:MATCHING} (where $2d$ plays the role of $d$), we find an independent set $\mathcal{X}$ of size $2t$ in $G$ such that $|\mathcal{X}\cap \mathcal{Q}_i|=2$ for all $i\in [t]$. Letting $\mathcal{X}\cap \mathcal{Q}_i=\{Q_{ia_i},Q_{ib_i}\}$, we see that \ref{item:Q-is-partial-decomposition} holds as $\mathcal{X}$ is an independent set in $G$. Moreover, {we have $z_{ia_i}\neq z_{ib_i}$ for all $i\in [t]$. Then,} \ref{item:endpoints-same-sign} holds by the definition of $\mathcal{Q}_i$, so the claim follows.\\
		
		\no Using \ref{item:Q-is-partial-decomposition}, $A$ has a perfect path decomposition $\mathcal{P}_0$ with $\mathcal{X}\subseteq \mathcal{P}_0$ by \Cref{remark:ACYCLIC-COMPLETION}. Then, 
		$\mathcal{P}_0$ is a partial path decomposition for $D$ as $\mathrm{ex}_D(v)=\mathrm{ex}_A(v)$ for all $v\in V(D)$, and hence so is $\mathcal{P}_0\setminus \mathcal{X}$. By \Cref{prop:PARTIAL-COMPLETION}, it suffices to prove that the digraph $D'$ induced by the edges $\bigcup_{i\in [t]}\left(E(C_i)\cup E(Q_{ia_i})\cup E(Q_{ib_i})\right)$ is consistent. It is clear that $\mathrm{ex}(D')=2t$, and we will show that for each $i\in [t]$, $C_i\cup \{Q_{ia_i},Q_{ib_i}\}$ can be decomposed into two paths. Fix $i\in [t]$. Using \ref{item:endpoints-same-sign}, we have either $z_{ia_i},z_{ib_i}\in V^+(D)$ or $z_{ia_i},z_{ib_i}\in V^-(D)$. Then, if $z_{ia_i},z_{ib_i}\in V^+(D)$, we define $U_i:=z_{ia_i}C_iz_{ib_i}Q_{ib_i}r(Q_{ib_i})$ and $V_i:=z_{ib_i}C_iz_{ia_i}Q_{ia_i}r(Q_{ia_i})$. Similarly, if $z_{ia_i},z_{ib_i}\in V^-(D)$, we define $U_i:=\ell(Q_{ia_i})Q_{ia_i}z_{ia_i}C_iz_{ib_i}$ and $V_i:=\ell(Q_{ib_i})Q_{ib_i}z_{ib_i}C_iz_{ia_i}$. In either case, $\{U_i,V_i\}$ decomposes $C_i\cup \{Q_{ia_i},Q_{ib_i}\}$, so the result follows.
	\end{proof}	
	
	%%	VIRESH'S SUGGESTION
	%		\noindent The next proposition says that for any digraph $D$, we can find a maximal cycle family in which every cycle is chordless, i.e.~every cycle is an induced subgraph of $D$.
	\no The following proposition gives us a cycle family with a simple but crucial structural property that is needed for our absorption method.
	%\no The following proposition will be crucial at the start of the proofs of \Cref{theorem:LARGE-GIRTH,theorem:KEY}.
	%\noindent For any digraph, we can find a maximal cycle family such that for any edge not belonging to any of these cycles, its endpoints do not lie on the same cycle. {CHECK THE STATEMENT OF THE PROPOSITION}
	
	\begin{proposition}\label{prop:CYCLE-SELECTION}
		Let $D$ be digraph. Then, there exists a maximal cycle family $\mathcal{C}$ in $D$ such that, writing $A:=D-\mathcal{C}$, we have $xy\notin E(A)$ for every $C\in \mathcal{C}$ and distinct $x,y\in V(C)$.   
	\end{proposition}
	\begin{proof}
		Let $t\geq 0$ be the maximum integer such that $D$ has $t$ many edge-disjoint cycles. Then, take a cycle family $\mathcal{C}$ with $|\mathcal{C}|=t$ such that $\sum_{C\in \mathcal{C}}|E(C)|$ is as small as possible. Since $\mathcal{C}$ has $t$ cycles, it is clear that $A$ is acyclic.  {Assume for a contradiction that there exists a cycle $C_0\in \mathcal{C}$ and distinct vertices $x,y\in V(C_0)$ such that $xy\in E(A)$. As there are no double edges, we have $xy\notin E(C_0)$.} Let $C_1$ be the cycle obtained from $C_0$ by replacing the arc $xC_0y$ by the edge $xy$. It is clear that $|E(C_1)|<|E(C_0)|$. Then, $(\mathcal{C}-C_0)\cup C_1$ is another cycle family consisting of $t$ many cycles, having strictly less than $\sum_{C\in \mathcal{C}}|E(C)|$ edges in total, a contradiction.
	\end{proof}
	
	\no We will use the following notation in the proofs of \Cref{theorem:LARGE-GIRTH,theorem:KEY}. For a cycle $C$ and a plus-minus path $P$ that is incident to $C$ at $v$ in a digraph $D$, we define the \textit{precise part of $P$ with respect to $C$ through $v$}, denoted by $\mathrm{pp}(P,C,v)$, as follows. If $v\in V^+(D)$ (so if $P$ starts at $v$), we define $\mathrm{pp}(P,C,v)$ as $zPy$ where $y$ is the first minus vertex along $P$ and $z\neq y$ is the last vertex along $vPy$ lying in $C$ (possibly $z=v$). Similarly, if $v\in V^-(D)$ (so if $P$ ends at $v$), we define $\mathrm{pp}(P,C,v)$ as $yPz$ where $y$ is the last plus vertex along $P$ and $z\neq y$ is the first vertex along $yPv$ lying in $C$ (possibly $z=v$). Note by the definition that $\mathrm{pp}(P,C,v)\subseteq P$, and that either $\mathrm{pp}(P,C,v)$ is tangent to $C$ (if $y\notin V(C)$) or $V(\mathrm{pp}(P,C,v))\cap V(C)=\{y,z\}$. Moreover, $\mathrm{pp}(P,C,v)\subseteq P$ is a plus-minus path in $D$ unless $z\in V^0(D)$, so $\mathrm{pp}(P,C,v)$ will always be a plus-minus path in the proof of \Cref{theorem:LARGE-GIRTH}.
	
	%\no For a digraph $D$ and a cycle $C$ in $D$, let $P$ be a plus-minus path that is incident to $C$ at $v$ where $E(C)\cap E(P)=\emptyset$. We define the \textit{tangent path of $P$ and $C$ through $v$}, denoted by $tp(P,C,v)$, as follows. If $v\in V^+(D)$, we define $tp(P,C,v)$ as $zPy$ where $y$ be the first minus vertex along $P$ and $z$ is the last vertex along $vPy$ lying in $C$ (possibly $z=v$). Similarly, if $v\in V^-(D)$, we define $tp(P,C,v)$ as $yPz$ where $y$ is the last plus vertex along $P$ and $z$ is the first vertex along $yPv$ lying in $C$ (possibly $z=v$). Note that $tp(P,C,v)\subseteq P$ is a plus-minus path if $z\notin V^0(D)$, and that $E(tp(P,C,v))\cap E(C)=\emptyset$. Moreover, either $tp(P,C,v)$ is tangent to $C$ (if $y\notin V(C)$) or $V(tp(P,C,v))\cap V(C)=\{y,z\}$. Now, we are ready to prove \Cref{theorem:LARGE-GIRTH,theorem:KEY}.
	
	\begin{proof}[Proof of \Cref{theorem:LARGE-GIRTH}]
		If $D$ is acyclic, then the result follows by \Cref{prop:acyclic-consistent}. Otherwise, using \Cref{prop:CYCLE-SELECTION}, let us take a cycle family $\mathcal{C}=\{C_i:i\in [t]\}$ in $D$ such that $A:=D-\mathcal{C}$ is acyclic, and
		\begin{align}
			xy\notin E(A)\text{ for all }i\in [t]\text{ and distinct }x,y\in V(C_i). \label{eqn:no-chord-property} \tag{*}
		\end{align}
		By \Cref{prop:acyclic-consistent}, $A$ is consistent. Let $\mathcal{P}=\{P_j:j\in [r]\}$ be a perfect path decomposition for $A$ where $r:=\mathrm{ex}(A)=\mathrm{ex}(D)$. We define an undirected bipartite graph $G=G(\mathcal{C},\mathcal{P})$ such that for any $i\in[t]$ and $j\in [r]$, we have $C_iP_j\in E(G)$ if and only if $P_j$ is incident to $C_i$. Note that $d_{G}(P_j)\leq 2d$ for every $j\in [r]$ as $\Delta^0(D)\leq d$. Moreover, by Observation~\ref{prop:basic}, for every $i\in [t]$, we have $d_G(C_i)\geq g(D)\geq 200d^2$ as $V^0(D)=\emptyset$. Hence, by \Cref{prop:HALL-CONSEQUENCE}, there exists a collection $\{\mathcal{R}_i:i\in [t]\}$ of pairwise disjoint subsets of $[r]$ such that for all $i\in [t]$ and $j\in \mathcal{R}_i$, we have $|\mathcal{R}_i|=100d$, and $P_j$ is incident to $C_i$. For notational convenience, for any $i\in [t]$ and $j\in \mathcal{R}_i$, write $P_{ij}=P_j$. Note that $\{P_{ij}:i\in [t],\, j\in \mathcal{R}_i\}$ is a path family. Hence, by \Cref{prop:KEY}, it suffices to find plus-minus paths $Q_{ij}\subseteq P_{ij}$ such that $Q_{ij}$ is tangent to $C_i$ for all $i\in [t]$ and $j\in \mathcal{R}_i$. Fix $i\in [t]$ and $j\in \mathcal{R}_i$, and take $v_{ij}\in V(C_i)$ such that $P_{ij}$ is incident to $v_{ij}$. Let $\ell(\mathrm{pp}(P_{ij},C_i,v_{ij}))=x$ and $r(\mathrm{pp}(P_{ij},C_i,v_{ij}))=y$. Note that $xP_{ij}y$ is a plus-minus path as $V^0(D)=\emptyset$, and that $E(xP_{ij}y)\cap E(C_i)=\emptyset$. Moreover, either $xP_{ij}y$ is tangent to $C_i$ or $V(C_i)\cap V(xP_{ij}y)=\{x,y\}$. In the former case, we define $Q_{ij}$ to be $xP_{ij}y$. In the latter case, by~\eqref{eqn:no-chord-property}, there exists a vertex $z\in V(xP_{ij}y)$ other than $x$ and $y$. If $z\in V^+(D)$, we define $Q_{ij}$ to be $zP_{ij}y$, and if $z\in V^-(D)$, we define $Q_{ij}$ to be $xP_{ij}z$. In every case, $Q_{ij}$ is a plus-minus path that is tangent to $C_i$, so the result follows.
		%\no If $v_{ij}\in V^+(D)$, let $y$ be the first minus vertex along the path $P_{ij}$. If $y\notin V(C_i)$, we define $Q_{ij}:=zP_{ij}y$ where $z$ is the last vertex along $v_{ij}P_{ij}y$ lying in $C_i$ (possibly $z=v_{ij}$). We remark that $z\in V^+(D)$ by the definition of $y$. If $y\in V(C_i)$, by~\eqref{eqn:no-chord-property}, the path $zP_{ij}y$ has at least one vertex other than $z$ and $y$. In that case, we define $Q_{ij}:=xy$ where $x$ is the unique inneighbour of $y$ along $P_{ij}$. We remark that $x\in V^+(D)$ by the definition of $y$. Similarly, if $v_{ij}\in V^-(D)$, let $y$ be the last plus vertex along the path $P_{ij}$. If $y\notin V(C_i)$, we define $Q_{ij}:=yP_{ij}z$ where $z$ is the first vertex along $yP_{ij}v_{ij}$ lying in $C_i$ (possibly $z=v_{ij}$). We remark that $z\in V^-(D)$ by the definition of $y$. If $y\in V(C_i)$, by~\eqref{eqn:no-chord-property}, the path $yP_{ij}z$ has at least one vertex other than $y$ and $z$. In that case, we define $Q_{ij}:=yx$ where $x$ is the unique outneighbour of $y$ along $P_{ij}$. We remark that $x\in V^-(D)$ by the definition of $y$. In every case, $Q_{ij}$ is a plus-minus path that is tangent to $C_i$ by construction, the result follows.
	\end{proof}
	
	\no We need a simple observation for the proof of \Cref{theorem:KEY}. Recall the definition of a $k$-sparse set at the beginning of this section.
	
	\begin{observation}\label{observation:sparse}
		Let $D$ be a digraph, and $S$ be a $k$-sparse set in $D$ for some $k\in \mathbb{N}$. Then, we have $|V(C)\cap S|<2|V(C)|/k$ for any cycle $C$ in $D$.
	\end{observation}
	\begin{proof}
		Let $\{a_1,\ldots,a_t\}$ be the vertices from $S$ lying on $C$ in this order clockwise. Letting $b_i:=|a_iCa_{i+1}|$, we have $b_i+b_{i+1}>k$ for each $i\in [t]$ (with the convention $a_{t+1}=a_1$ and $b_{t+1}=b_1$), because otherwise we would have $\mathrm{dist}_{D^\diamond}(a_i,a_{i+1})+\mathrm{dist}_{D^\diamond}(a_{i+1},a_{i+2})\leq k$, which contradicts $S$ being $k$-sparse. Hence, the result follows by summing up all these inequalities for  {$i\in [t]$}.
	\end{proof}
	
	\begin{proof}[Proof of \Cref{theorem:KEY}]
		If $D$ is acyclic, then the result follows by \Cref{prop:acyclic-consistent}. Otherwise, using \Cref{prop:CYCLE-SELECTION}, let us take a cycle family $\mathcal{C}=\{C_i:i\in [t]\}$ in $D$ such that $A:=D-\mathcal{C}$ is acyclic, and~\eqref{eqn:no-chord-property} holds. By \Cref{prop:acyclic-consistent}, $A$ is consistent. Let $\mathcal{P}=\{P_j:j\in [r]\}$ be a perfect path decomposition for $A$ where $r:=\mathrm{ex}(A)=\mathrm{ex}(D)$. For each $i\in [t]$, we will construct a subset $\mathcal{R}_i\subseteq \mathcal{P}$ such that 
		\begin{enumerate}[label={\rm (\alph*$'$)}]
			\item $|\mathcal{R}_i|\geq 200d^2$, \label{item:R-i-large}
			\item for each $P_j\in \mathcal{R}_i$, there is a plus-minus path $\mathcal{Q}_{ij}\subseteq P_j$ that is tangent to $C_i$. \label{item:tangency}
		\end{enumerate}
		
		\no Fix $i\in [t]$. Write  {$V^{\rho}:=V^{\rho}(D)\cap V(C_i)$} for $\rho\in \{+,-,0\}$ for simplicity. Without loss of generality, assume that $|V^+|\geq|V^-|$. Let $V^{b}:=\{v\in V^+:N^+_D(v)\cap V^0(D)\neq \emptyset\}$. Let $\mathcal{R}_i$ be the set of paths $P_j\in \mathcal{P}$ such that $V(P_j)\cap (V^0\cup V^b)=\emptyset$ and $\ell(P_j)\in V^+$. Note that each $P_j\in \mathcal{R}_i$ is a plus-minus path incident to $C_i$ at $\ell(P_j)$. We will prove that \ref{item:R-i-large} holds. Using Observation~\ref{prop:basic}, there are at least $|V^+|$ paths $P\in \mathcal{P}$ with $\ell(P)\in V^+$. Using $\Delta^0(D)\leq d$, each $v\in V^0\cup V^b$ lies on at most $d$ paths in $\mathcal{P}$, so it suffices to show that $|V^+|-d|V^0\cup V^b|\geq 200d^2$. Let $|V(C_i)|=g$. By Observation~\ref{observation:sparse}, we have $|V^0|<2g/k$ as $V^0(D)$ is $k$-sparse. Let $p$ be the maximum integer such that there exist distinct $x_1,\ldots,x_p\in V^b$ and distinct $z_1,\ldots,z_p\in V^0(D)$ such that $x_jz_j\in E(D)$ for all $j\in [p]$. Note that 
		$|V^b|\leq dp$ as $\Delta^0(D)\leq d$. Since $\mathrm{dist}_{D^\diamond}(z_j,z_{j'})\leq \mathrm{dist}_{D^\diamond}(x_j,x_{j'})+2$ for all distinct $j,j'\in [p]$, we see that $\{x_1,\ldots,x_p\}$ is a $(k-2)$-sparse set in $D$. By Observation~\ref{observation:sparse}, we have $p<2g/(k-2)$. Therefore, we obtain $|V^b|<2gd/(k-2)$. Hence, we have
		\begin{align*}
			|V^+|-d|V^0\cup V^b|\geq \dfrac{g-|V^0|}{2}-d(|V^0|+ |V^b|)> g\cdot\left(\dfrac{1}{2}-\dfrac{2d+1}{k}-\dfrac{2d^2}{k-2}\right)\geq 200d^2,
		\end{align*}
		where the last inequality holds since $g\geq 1000d^2$ and $k\geq 20d^2+2$.\\

		\no Next, we will show that \ref{item:tangency} holds.   Fix some $P_j\in \mathcal{R}_i$. Let us write $X_{j}:=\mathrm{pp}(P_j,C_i,\ell(P_j))$, $\ell(X_{j})=u_{j}$ and $r(X_{j})=w_{j}$. Recall that $w_j$ is the first minus vertex along $P_j$ and $u_j\neq w_j$ is the last vertex along $\ell(P_j)P_jw_j$ lying in $C_i$ (possibly $u_j=\ell(P_j)$). Since $V(P_j)\cap V^0=\emptyset$, we have $u_j\in V^+$. Now, if $w_j\notin V(C_i)$, then $X_j$ is tangent to $C_i$. In that case, we define $\mathcal{Q}_{ij}$ to be $X_j$. Otherwise, by \eqref{eqn:no-chord-property}, we have $|X_j|\geq 2$. Let $z_j$ be the unique outneighbour of $u_j$ along $X_j$. Since $V(P_j)\cap V^b=\emptyset$, we have $z_j\notin V^0(D)$. If  {$z_j\in V^+(D)$}, we define $\mathcal{Q}_{ij}$ to be $z_jP_jw_j$, and if  {$z_j\in V^-(D)$}, we define $\mathcal{Q}_{ij}$ to be $u_jz_j$. In all cases, we have $\mathcal{Q}_{ij}\subseteq P_j$ is a plus-minus path that is tangent to $C_i$. \\

		\no Let us construct an undirected bipartite graph $G=G([t],[r])$ such that for all $i\in [t]$ and $j\in [r]$, $ij\in E(G)$ if and only if $P_j\in \mathcal{R}_i$. By the definition of $\mathcal{R}_i$, $ij\in E(G)$ implies $P_j$ is incident to $C_i$, so we have $d_G(j)\leq 2d$ for all $j\in [r]$. On the other hand, by \ref{item:R-i-large}, we have $d_G(i)\geq 200d^2$ for all $i\in [t]$. Hence, by  \Cref{prop:HALL-CONSEQUENCE}, there exists a collection $\{\mathcal{S}_i:i\in [t]\}$ of pairwise disjoint subsets of $\mathcal{P}$ such that $\mathcal{S}_i\subseteq \mathcal{R}_i$ and $|\mathcal{S}_i|=100d$ for all $i\in [t]$. Then, by \ref{item:tangency}, for all $P_j\in \mathcal{S}_i$, there exists a plus-minus path $\mathcal{Q}_{ij}\subseteq P_j$ that is tangent to $C_i$. Since $\{P_j:j\in [r]\}$ is a path family, we can conclude that $\{\mathcal{Q}_{ij}:i\in [t],\, P_j\in \mathcal{S}_i\}$ is a path family. Hence, the result follows by \Cref{prop:KEY}.
	\end{proof}

	\section{Random regular graphs}\label{sec:random-regular}
	
	\no In this section we prove \Cref{thm:RANDOM-REGULAR} by showing a more general result, namely \Cref{thm:discrete-orientation}. We first show that a random $d$-regular graph has certain properties with high probability (see \Cref{def:n-d-k-discrete} and \Cref{prop:RANDOM-REGULAR}). Then, we prove that for every graph with these properties, every orientation without excess-zero vertices is consistent, where the main ingredient of the proof uses \Cref{prop:PM(D)-relating-SD(v)-V0(D)} and \Cref{theorem:KEY}.\\
	
	\no Throughout the section, we use standard probability theory notation and terminology. We say a random variable $X$ has \textit{Poisson distribution with parameter $\lambda$} if $\Pro(X=k)=\lambda^k e^{-k}/k!$ for all nonnegative integers $k$, denoted by $X\sim \mathrm{Po}(\lambda)$. It is well-known that $\E[X]=\lambda$ for such a random variable. Moreover, for independent random variables $X_i\sim \mathrm{Po}(\lambda_i)$, we have $\sum_{i}X_i\sim \mathrm{Po}(\lambda)$ where $\lambda=\sum_{i}\lambda_i$. Given $n$ and $d$, we denote the set of all $d$-regular graphs on $n$ vertices by $\mathbb{G}_{n,d}$. We write $\mathcal{G}_{n,d}$ for the random $d$-regular graph on $n$ vertices, that is a graph uniformly selected from $\mathbb{G}_{n,d}$. We denote the set of all $d$-regular \textit{multigraphs} on $n$ vertices by $\mathbb{S}_{n,d}$, where we allow loops and more than one edges between two vertices. The \textit{configuration model}, introduced by Bollobás \cite{BOLLOBAS}, provides a model for such graphs {. Each partition $\mathscr{P}$ of $[n]\times [d]$ into $nd/2$ pairs (equivalently each perfect matching in $K_{nd}$) naturally corresponds to a graph $G$ in $\mathbb{S}_{n,d}$ by including} an edge between $a\in [n]$ and $a'\in [n]$  {in $G$ if and only if} $(a,b),(a',b')\in [n]\times [d]$ are matched with each other  {in $\mathscr{P}$} for some $b,b'\in [d]$. We write $\mathcal{S}_{n,d}$ for the random $d$-regular multigraph on $n$ vertices  {corresponding to a pairing of $[n] \times [d]$} sampled  {uniformly at random}. Given a (di)graph property $\mathrm{Q}$ and a sequence of random (di)graphs $\{G_n\}_{n>0}$ with $|V(G_n)|\to \infty$ as $n\to \infty$, we say \textit{$G_n$ has $\mathrm{Q}$ with high probability} if $\Pro(G_n\text{ has }\mathrm{Q})\to 1$ as $n\to \infty$. The configuration model is quite useful to work on $\mathcal{G}_{n,d}$ because of the following folklore result.
	\begin{proposition}[see e.g.~\cite{BOOK-BOLLOBAS} {, Corollary 2.18}]\label{prop:Configuration}
		For given  {$d\geq 2$}, any graph property that holds with high probability for $\mathcal{S}_{n,d}$ holds with high probability for $\mathcal{G}_{n,d}$.
	\end{proposition}
	
	%	\no For a real-valued random variable $X$, its \textit{cumulative distribution function} $F_X(x)$ is defined as $F_X(x)=\Pro(X\leq x)$ for all $x\in \mathbb{R}$. We say a sequence $\{X_n\}_{n>0}$ of random variables \textit{converges in distribution} to a random variable $X$ if $\lim_{n\to \infty}F_{X_n}(x)=F_X(x)$ for all $x$ at which $F_X(x)$ is continuous.

	%	\subsection{Random Regular Graphs}\label{subsec:random-regular}
	
	%\no In this section, we discuss the key properties of a random regular graph used in \Cref{sec:random-regular}. 
	\no For a graph $G$, let $Y_i(G)$ denote the number of cycles of length $i$. The number of cycles of a fixed length in a random regular graphs is known to have a Poisson distribution.
	%Let $Z_k$ denote a random variable that has Poisson distribution with parameter $\lambda_k:=(d-1)^k/2k$. 
	\begin{proposition}[see e.g.~\cite{BOOK-BOLLOBAS} {, Corollary 2.19}]\label{prop:RANDOM-REGULAR-N-CYCLES}
		For fixed $k\geq 3$ and $d\geq 2$, let $\{Z_i:3\leq i\leq k\}$ be independent random variables with $Z_i\sim \mathrm{Po}((d-1)^i/2i)$, and let $Y_i^n:=Y_i(\mathcal{G}_{n,d})$ for $3\leq i\leq k$. Then, $(Y_3^n,\ldots,Y_k^n)_{n>0}$ converges in distribution to $(Z_3,\ldots,Z_k)$.
	\end{proposition}
	
	\no Recall Markov's inequality which states that for a nonnegative random variable $X$ and a constant $a>0$, we have $\Pro(X\geq a)\leq \E[X]/a$. It is easy to see that $\mathcal{G}_{n,d}$ contains only a few short cycles.
	
	\begin{proposition}\label{cor:RANDOM-REGULAR-BOUNDED-CYCLES}
		For fixed $k\geq 3$ and $d\geq 2$, as $n\to \infty$, we have
		\begin{align*}
			\Pro\left(\sum_{i=3}^{k}Y_i(\mathcal{G}_{n,d})\leq \log \log \log n\right) \to 1.
		\end{align*}
	\end{proposition}
	
	\begin{proof}
		Let $Y:=\sum_{i=3}^{k}Y_i(\mathcal{G}_{n,d})$ and $Z=\sum_{i=3}^k Z_i$ where $\{Z_i:3\leq i\leq k\}$ are independent random variables with $Z_i\sim \mathrm{Po}((d-1)^i/2i)$. Note that $Z\sim \mathrm{Po}(\lambda)$ where $\lambda:=\sum_{i=3}^{k}(d-1)^i/2i$. Since $k$ is fixed, by \Cref{prop:RANDOM-REGULAR-N-CYCLES}, $Y$ converges in distribution to $Z$.  Using Markov's inequality, we have $\Pro(Z\geq \log\log \log n)\leq \lambda /\log\log\log n$, so $\Pro(Z\geq \log\log \log n)\to 0$ as $n\to \infty$ as $k$ and $d$ are fixed. Since $\Pro(Y\geq \log\log\log n)\to \Pro(Z\geq \log\log\log n)$, the result follows. 
	\end{proof}
	
	%\no Recall that if $Y=\sum_{i\in [t]}Y_i$ where $Y_i$s are Poisson distributed random variables with parameters $\rho_i$, then $Y$ is a Poisson distributed random variable with parameter $\sum_{i\in [t]}\rho_i$. Recall the Bennett's inequality which states that if $Y$ is a Poisson random variable with parameter $\rho$, then $\Pro(Y\geq \rho(1+u))\leq e^{-\rho h(u)}$ for $u\geq 0$ where $h(u):=(1+u)\log (1+u)-u$. Note that $h(u)\geq \log (1+u)$ for $u\geq e-1$, so we have $\Pro(Y\geq \rho s)\leq s^{-\rho}$ for $s\geq e$. For a graph $G$, let $Y'_k(G)$ be the number of cycles of length at most $k$. Note that $\sum_{i=3}^{k}(d-1)^i/2i=o(\log n)$ for any fixed $k$. Then, the following is just a consequence of the discussion above. 
	
	\no Next, we give a bound for the probability of the existence of a fixed graph in $\mathcal{S}_{n,d}$. 
	
	\begin{proposition}\label{prop:RANDOM-SMALL-SUBGRAPH}
		Let $k\in \mathbb{N}$. For a graph $G_0$ with $|V(G_0)|=k$ and $|E(G_0)|=k+1$, we have
		\begin{align*}
			\mathbb{P}(\mathcal{S}_{n,d}\text{ has a copy of }G_0)\leq (2d)^{k+1}/n   
		\end{align*}
		for all natural numbers $d<n$  {with $dn$ even and $dn\geq 4k+2$}.
	\end{proposition}
	\begin{proof}
		Recall that there are $d$ copies of  {each vertex} in the configuration model.  {For each vertex $v\in [n]$, let $v^j$ denote the $j$th copy of $v$ for $j\in [d]$. For simplicity, we write $[u^i-v^j]$ to mean $u^i$ is paired up with $v^j$. Note that for distinct $u,v\in [n]$, we have $uv\in \mathcal{S}_{n,d}$ if and only if $[u^{i}-v^j]$ for some $i,j\in [d]$. Let $p:=dn/2$, and consider a fixed collection of edges $\{e_1,\ldots,e_{k+1}\}$ on the vertex set $[n]$. For $i\in [k+1]$, let us write $e_i=x_iy_i$ for some $x_i,y_i\in [n]$. Given a complete graph on $2r$ vertices, it is well-known that the number of pairings is $(2r-1)(2r-3)\cdots 3 \cdot 1$. Hence, using the union bound, we can write
			\begin{align*}
				\Pro(e_1,\ldots,e_{k+1}\in E(\mathcal{S}_{n,d}))&= \Pro([x_1^{i_1}-y_1^{j_1}],\ldots, [x_{k+1}^{i_{k+1}}-y_{k+1}^{j_{k+1}}]\text{ for some }i_1,j_1,\ldots,i_{k+1},j_{k+1}\in [d])\\
				&=\Pro\left(\bigcup_{i_1,j_1,\ldots,i_{k+1},j_{k+1}\in[d]} \left([x_1^{i_1}-y_1^{j_1}] \wedge \ldots \wedge [x_{k+1}^{i_{k+1}}-y_{k+1}^{j_{k+1}}]\right) \right)\\
				&\leq \sum_{i_1,j_1,\ldots,i_{k+1},j_{k+1}\in[d]} \Pro \left([x_1^{i_1}-y_1^{j_1}] \wedge \ldots \wedge [x_{k+1}^{i_{k+1}}-y_{k+1}^{j_{k+1}}] \right)\\
				&\leq d^{2k+2} \cdot \dfrac{(2p-2k-3)(2p-2k-5)\cdots 3\cdot 1}{(2p-1)(2p-3)\cdots 3\cdot 1}\\
				&=d^{2k+2}\cdot\left((2p-1)(2p-3)\cdots(2p-2k-1)\right)^{-1}\\
				&= \dfrac{d^2}{dn-1}\cdot \dfrac{d^2}{dn-3} \cdots \dfrac{d^2}{dn-2k-1}\\
				&\leq (2d/n)^{k+1},
			\end{align*}
			where the last inequality holds as $dn\geq 4k+2$.
			%So, we  have $\Pro(e_1,\ldots,e_{k+1}\in E(\mathcal{S}_{n,d}))\leq (2d/n)^{k+1}$.
			Then, as there are at most $\binom{n}{k}k!$ copies of $G_0$ in the complete graph on $n$ vertices,} the expected number of copies of $G_0$ in $\mathcal{S}_{n,d}$ can be bounded above by $\binom{n}{k}k!(2d/n)^{k+1}\leq (2d)^{k+1}/n$. The result follows by Markov's inequality.
	\end{proof}
	
	\begin{definition}\label{def:n-d-k-discrete}
		For $n,d,p\in \mathbb{N}$, we say a graph $G$ is \textit{$(n,d,p)$-discrete} if $|V(G)|=n$, $\Delta(G)\leq d$, and, letting $\mathcal{C}$ be the set of all cycles of length at most $p$, the following hold:
		\begin{enumerate}[label={\rm (\roman*)}]
			\item $|\mathcal{C}|\leq \log\log \log n$; \label{item:N-SHORT-CYCLES}
			\item $V(C)\cap V(C')=\emptyset$ for all distinct $C,C'\in \mathcal{C}$; \label{item:CYCLES-DISJOINT}
			\item $\mathrm{dist}_{G-\mathcal{C}}(u,v)\geq \log \log n$ for all distinct $u,v\in V(\mathcal{C})$.  \label{item:CYCLES-FAR-AWAY}
		\end{enumerate}
	\end{definition}
	\no We prove that a $d$-regular random graph on $n$ vertices is $(n,d,p)$-discrete with high probability.
	
	\begin{proposition}\label{prop:RANDOM-REGULAR}
		For fixed $p\geq 3$ and $d\geq 2$, $\mathcal{G}_{n,d}$ is $(n,d,p)$-discrete with high probability. 
	\end{proposition}
	\begin{proof}
		By \Cref{cor:RANDOM-REGULAR-BOUNDED-CYCLES}, we know that property \ref{item:N-SHORT-CYCLES} holds for $\mathcal{G}_{n,d}$ with high probability. Let $n$ be a sufficiently large natural number and $\mathcal{C}$ be the collection of all cycles in $\mathcal{G}_{n,d}$ of length at most $p$. For {$x,y\geq 2$ and $z\geq 1$}, $F_{x,y,z}$ denotes the (unique) unlabelled graph obtained from three internally vertex disjoint paths of length $x$, $y$, $z$ between two fixed vertices. Similarly, for {$x,y\geq 3$} and $z\geq 0$, $H_{x,y,z}$ denotes the (unique) unlabelled graph obtained from two cycles of length $x$ and $y$ by taking one vertex from each cycle and joining them via a path of length $z$. Note that each $F_{x,y,z}$ and $H_{x,y,z}$ has $x+y+z-1$ vertices and $x+y+z$ edges. Let $\mathcal{F}=\{F_{x,y,z}:x,y\leq p,\,z\leq \log \log n\}$ and {$\mathcal{H}=\{H_{x,y,z}:x,y\leq p,\,z\leq \log \log n\}$}. Observe that if two distinct cycles $C,C'\in \mathcal{C}$ have a unique common vertex, then $\mathcal{G}_{n,d}$ has a copy of $H_{|C|,|C'|,0}$. If $C,C'\in \mathcal{C}$ have more than one vertices in common, then $\mathcal{G}_{n,d}$ has a copy $F_{a,b,c}$ for some $a,b,c\leq p$. Similarly, if $\mathrm{dist}_{\mathcal{G}_{n,d}-\mathcal{C}}(u,v)\leq \log \log n$ for some distinct $u,v\in V(\mathcal{C})$, it is easy to see that $\mathcal{G}_{n,d}$ contains a graph from $\mathcal{F}\cup \mathcal{H}$. As a result, if $\mathcal{G}_{n,d}$ contains no graph from $\mathcal{F}\cup \mathcal{H}$, then properties \ref{item:CYCLES-DISJOINT} and \ref{item:CYCLES-FAR-AWAY} hold. Also note that {$|\mathcal{F}\cup \mathcal{H}|\leq p^2\log\log n+p^2(\log\log n+1)\leq 3p^2\log\log n$}, and that each graph in $\mathcal{F}\cup \mathcal{H}$ has at most $2p+\log\log n-1\leq 2\log\log n-1$ vertices. Hence, by \Cref{prop:RANDOM-SMALL-SUBGRAPH}, we obtain
		\begin{align*}
			\Pro(\mathcal{S}_{n,d}\text{ contains a graph from }\mathcal{F}\cup \mathcal{H})\leq 3p^2\log\log n(2d)^{2\log\log n}/n\leq 1/\sqrt{n},
		\end{align*}
		where the last inequality holds for sufficiently large $n$.
		Therefore, with high probability, $\mathcal{S}_{n,d}$ does not contain any graph from $\mathcal{F}\cup \mathcal{H}$. By \Cref{prop:Configuration}, the same holds for $\mathcal{G}_{n,d}$, so properties \ref{item:CYCLES-DISJOINT} and \ref{item:CYCLES-FAR-AWAY} hold with high probability for $\mathcal{G}_{n,d}$, which completes the proof.
	\end{proof}
	
	\no We need the following definitions in the proofs of \Cref{prop:PM(D)-relating-SD(v)-V0(D)} and \Cref{thm:discrete-orientation}. We say a path family $\mathcal{Q}$ (in a digraph) of plus-minus paths is \textit{distinctive} if $\ell(Q)\neq \ell(Q')$ and $r(Q)\neq r(Q')$ for all distinct $Q,Q'\in \mathcal{Q}$. For a digraph $D$ and a vertex $v\in V(D)$, we define the set $\mathrm{PM}(D,v)$ as the set of vertices $w$ for which there exists a plus-minus path $P$ in $D$ such that $\{\ell(P),r(P)\}=\{v,w\}$. Note that $\mathrm{PM}(D,v)\neq \emptyset$ for all $v\in V(D)\setminus V^0(D)$. We also define the \textit{distance to sign change} of a vertex $v\in V(D)\setminus V^0(D)$, denoted by $\mathrm{SD}(D,v)$, as the minimum length of a plus-minus path $P$ in $D$ that is incident to $v$. The following proposition helps us to find a lower bound for $|\mathrm{PM}(D,v)|$ when $|V^0(D)|$ is small compared to $\mathrm{SD}(D,v)$.

	\begin{proposition}\label{prop:PM(D)-relating-SD(v)-V0(D)}
		Let $d\geq 2$ be a natural number. Then, for every digraph $D$ with $\Delta^0(D)\leq d$ and $\mathrm{ex}(D)\neq 0$, and for every vertex $v\in V(D)\setminus V^0(D)$ with $\mathrm{SD}(D,v)>3|V^0(D)|$, we have
		\begin{align*}
			|\mathrm{PM}(D,v)|\geq \dfrac{1}{d}\left(1+\dfrac{1}{2d}\right)^{\mathrm{SD}(D,v)/6}.     
		\end{align*}
	\end{proposition}
	\begin{proof}
		Due to symmetry, we may assume without loss of generality that $v\in V^+(D)$. For simplicity, write $|V^0(D)|=\ell$ and $\mathrm{SD}(D,v)=t$. Let $A_0=\{v\}$, and for every integer $k\geq 0$, define
		\begin{align*}
			A_{k+1}=\left(N^+(A_k)\cap (V^+(D)\cup V^0(D))\right)\setminus B_k,
		\end{align*}
		where $B_k=\bigcup_{i=0}^{k}A_i$. For every integer $k\geq 1$, we now construct a set $E_k\subseteq E(D)$ as follows. By definition, for every $w\in A_{k}$, we can find an edge $w'w\in E(D)$ with $w'\in A_{k-1}$. We pick one such edge for every $w\in A_{k}$ and put it in $E_k$. Letting $T$ be the digraph induced by $\bigcup_{k\geq 1}E_k$, we can see by the construction that for every $w\in V(T)$, there exists a unique path $P$ with $\ell(P)=v$ and $r(P)=w$. In fact, $T$ is an oriented out-tree rooted at $v$.\\   
		
		\no If $t=1$, then $(1+1/2d)^{1/6}\leq (1+1/2d)\leq d$ as $d\geq 2$, so the result follows as $|\mathrm{PM}(D,v)|\geq 1$. Suppose $t\geq 2$. By the definition of $\mathrm{SD}(D,v)$, we have $|A_i|\geq 1$ for all $0\leq i\leq t-1$. This, in particular, implies that $|B_{2\ell}|\geq 2\ell+1$ as $2\ell<t$. Now, fix some $2\ell \leq i\leq t-2$; we will prove that $|B_{i+1}|\geq (1+1/2d)|B_i|$. Since $B_{i}\subseteq V^+(D)\cup V^0(D)$, we have $\sum_{v\in B_i}\mathrm{ex}(v)\geq |B_i|-\ell\geq |B_i|/2$ as $|B_i|\geq |B_{2\ell}|\geq 2\ell+1$. Hence, we can find $|B_{i}|/2$ edges $f\in E(D)$ with $\ell(f)\in B_i$ and $r(f)\notin B_i$. Therefore, since $i\leq t-2$, it is necessarily the case that $\ell(f)\in A_i$ and $r(f)\in A_{i+1}$. This implies $|A_{i+1}|\geq |B_i|/2d$ as $\Delta^0(D)\leq d$. Therefore, we obtain $|B_{i+1}|\geq (1+1/2d)|B_i|$. Inductively, we obtain
		\begin{align*}
			|B_{t-1}|\geq |B_{2\ell}|(1+1/2d)^{t-1-2\ell}\geq (2\ell+1)(1+1/2d)^{(t-1)/3}\geq \ell+(1+1/2d)^{t/6},   
		\end{align*}
		where the last inequality holds as $\ell\geq 0$ and $t\geq 2$, and the penultimate inequality holds as $\ell\leq (t-1)/3$ and $|B_{2\ell}|\geq 2\ell+1$. By the definition of $T$, note that for any edge $f\in E(D)$ with $\ell(f)\in V(T)$ and $r(f)\notin V(T)$, we have $r(f)\in V^-(D)$. Since $V(T)\subseteq V^+(D)\cup V^0(D)$ and $B_{t-1}\subseteq V(T)$, we have $\sum_{v\in V(T)}\mathrm{ex}(v)\geq |B_{t-1}|-\ell$. Then, we can find $|B_{t-1}|-\ell$ many edges $f\in E(D)\setminus E(T)$ with $\ell(f)\in V(T)$ and {$r(f)\notin V(T)$}. For every such edge, we have $r(f)\in \mathrm{PM}(D,v)$. Since $\Delta^0(D)\leq d$, we obtain 
		\begin{align*}
			|\mathrm{PM}(D,v)|\geq (|B_{t-1}|-\ell)/d\geq (1+1/2d)^{t/6}/d,
		\end{align*}
		so the result follows.
	\end{proof}

	\no Now, we are ready to prove our main result in this section.
	
	\begin{theorem}\label{thm:discrete-orientation}
		For natural numbers $d\geq 2$ and $p$ with $p\geq 1000d^2$, there exists an integer $n_0=n_0(d,p)$ such that the following holds. Let $G$ be an $(n,d,p)$-discrete graph with $n\geq n_0$. Then, every orientation $D$ of $G$ satisfying $\mathrm{ex}_D(v)\neq 0$ for all $v\in V(D)$ is consistent.
		%Let $D$ be a digraph on $n\geq n_0$ vertices with $\Delta^0(G)\leq d$ such that $\mathrm{ex}(v)\neq 0$ for all $v\in V(D)$. 
		%	Then, $D$ is consistent.
	\end{theorem}
	\begin{proof}
		Fix natural numbers $d\geq 2$ and $p\geq 1000d^2$. Let $k:=20d^2+2$, and $n$ be sufficiently large. Let $G$ be an $(n,d,p)$-discrete graph, and $D$ be an orientation of $G$ satisfying $\mathrm{ex}_D(v)\neq 0$ for all $v\in V(D)$. Let $\mathcal{C}_0$ be the collection of all cycles in $G$ of length at most $p$, and write $G_0:=G-\mathcal{C}_0$. Similarly, let $\mathcal{C}=\{C_i:i\in [t]\}$ be the collection of all directed cycles in $D$ of length at most $p$, and write $D_0:=D-\mathcal{C}$. Using \Cref{prop:PARTIAL-COMPLETION}, it suffices to find a partial path decomposition $\mathcal{P}$ for $D$ in which $D-\mathcal{P}$ is consistent. It is clear that $\Delta^0(D-\mathcal{P})\leq d$ and $V^0(D-\mathcal{P})\subseteq \bigcup_{P\in \mathcal{P}}\{\ell(P),r(P)\}$ for any partial path decomposition $\mathcal{P}$ for $D$. Then, by \Cref{theorem:KEY}, it is enough to find a partial path decomposition $\mathcal{P}$ for $D$ such that $g(D-\mathcal{P})\geq 1000d^2$ and that $\bigcup_{P\in \mathcal{P}}\{\ell(P),r(P)\}$ is $k$-sparse in $D - \mathcal{P}$. We will ensure the condition $g(D-\mathcal{P})\geq 1000d^2$ by choosing $\mathcal{P}$ in such a way that at least one edge from each $C_i\in \mathcal{C}$ is used in $\mathcal{P}$. As a first step, we prove that $\mathcal{C}$ satisfies certain properties as a consequence of $G$ being $(n,d,p)$-discrete.
		
		\begin{claim}\label{claim:D-is-n-d-p-discrete}
			The following hold:
			\begin{enumerate}[label={\rm (\roman*$'$)}]
				\item $t\leq \log\log\log n$; \label{item:DIRECTED-N-SHORT-CYCLES}
				\item for any distinct $C,C'\in \mathcal{C}$, $V(C)\cap V(C')=\emptyset$; \label{item:DIRECTED-CYCLES-DISJOINT}
				\item $\mathrm{dist}_{D_0^\diamond}(u,v)\geq \log \log n$ for all distinct $u,v\in V(\mathcal{C})$.\label{item:DIRECTED-CYCLES-FAR-AWAY}
			\end{enumerate}
		\end{claim} 
		\begin{proofclaim}
			Recall \Cref{def:n-d-k-discrete}. Since every cycle in $\mathcal{C}$ is an orientation of a cycle in $\mathcal{C}_0$, \ref{item:N-SHORT-CYCLES} and \ref{item:CYCLES-DISJOINT} immediately give \ref{item:DIRECTED-N-SHORT-CYCLES} and \ref{item:DIRECTED-CYCLES-DISJOINT}, respectively. Then, assume for a contradiction that, for some distinct $u,v\in V(\mathcal{C})$, there exists a path $P$ in $D_0^\diamond$ between $u$ and $v$ such that $|E(P)|<\log\log n$. By \ref{item:CYCLES-FAR-AWAY}, we have $\mathrm{dist}_{G_0}(u,v)\geq \log \log n$, which implies $P$ has an edge $xy\in E(\mathcal{C}_0)$ with $xy,yx\notin E(\mathcal{C})$. {Let $x_0y_0$ be} the first such edge {(i.e.~satisfying $x_0y_0\in E(\mathcal{C}_0)$ and $x_0y_0,y_0x_0\notin E(\mathcal{C})$)} encountered when traveling from $u$ to $v$ along $P$, with  {$x_0$} being closest to $u$. By \ref{item:CYCLES-DISJOINT}, we have {$x_0\in V(\mathcal{C}_0)\setminus V(\mathcal{C})$}, which in particular implies that  {$x_0\neq u$}. Then, we see that {$\mathrm{dist}_{G_0}(u,x_0)< \log \log n$ by the choice of $x_0y_0$}, which contradicts \ref{item:CYCLES-FAR-AWAY}. Hence, \ref{item:DIRECTED-CYCLES-FAR-AWAY} follows. 
		\end{proofclaim}
		
		\no Next, note that $\mathrm{ex}_{D_0}(v)=\mathrm{ex}_D(v)\neq 0$ for all $v\in V(D)$. Then, for every $i\in [t]$, we have either $|V(C_i)\cap V^+(D)|\geq 2$ or $|V(C_i)\cap V^-(D)|\geq 2$. Therefore, we can choose distinct $a_i,b_i\in V(C_i)$ for all $i\in [t]$ such that either $\{a_i,b_i\}\subseteq V^+(D)$ or $\{a_i,b_i\}\subseteq V^-(D)$. Recall the definition of distinctive path family just before \Cref{prop:PM(D)-relating-SD(v)-V0(D)}. 
		
		%\no Moreover, if there exists such $\mathcal{Q}$ with $|E(Q)|\leq L$ for all $Q\in\mathcal{Q}$, we say $T$ is \textit{$L$-nice}.
		\begin{claim}\label{claim:incident-partial-paths}
			There exists a distinctive path family $\mathcal{Q}$ in $D_0$ such that
			\begin{enumerate}[label={\rm (\arabic*)}]
				\item there exists a bijection $f:[t]\to\mathcal{Q}$ such that $\{\ell(f(i)),r(f(i))\}\cap V(\mathcal{C})=\{a_i\}$ for all $i\in [t]$, \label{item:Q-i-incident-a-i}
				%letting $\mathcal{Q}:=\{Q_i:i\in T\}$, for all $i\in T$, $Q_i$ is a (plus-minus) path in $D_0$ that is incident to $a_i$, \label{item:Q-i-incident-a-i}
				%	for all $i\in T$, $Q_i$ is a (plus-minus) path in $D_0$ that is incident to $a_i$, 
				\item $\mathrm{dist}_{D_0^\diamond}(u,v)>k$ for all distinct $u,v\in \left(\bigcup_{Q\in \mathcal{Q}} \{\ell(Q),r(Q)\}\right)\cup V(\mathcal{C})$ unless {possibly} $\{u,v\}=\{\ell(Q),r(Q)\}$ for some $Q\in \mathcal{Q}$. \label{item:endpoints-far-away}
			\end{enumerate}
		\end{claim}
		\begin{proofclaim}
			We begin with a brief explanation of our proof strategy. Note that if for all $i\in [t]$, we manage to find a ``short" plus-minus path $Q_i$ that is incident to $a_i$, then Claim~\ref{claim:D-is-n-d-p-discrete} implies \ref{item:endpoints-far-away}. We start the proof by taking the largest collection of such ``short" plus-minus paths $Q_i$. If we fail to find a path for some $a_i$, we can conclude that $a_i$ has large distance to sign change, which means there exist a ``large" number of plus-minus paths incident to $a_i$ using \Cref{prop:PM(D)-relating-SD(v)-V0(D)}. Hence, we have a lot of choice for the other endpoint of $Q_i$. In this way, \ref{item:endpoints-far-away} can be satisfied as there are only a ``small" number of vertices which are close to the endpoints of the paths previously chosen (as $\Delta(G) \leq d$).\\
			
			\no Set $\theta:=6(\log\log\log n+\log d)/\left(\log(2d+1)-\log 2d\right)$ so that 
			\begin{align}
				\dfrac{1}{d}\left(1+\dfrac{1}{2d}\right)^{\theta/6}=\log \log n. \label{eqn:choice-of-theta}
			\end{align}
			Since $n$ is sufficiently large, we have that 
			\begin{align}
				6\log\log\log n<\theta<(\log\log n)/3.\label{eqn:inequality-theta}
			\end{align}
			%such that there exists a distinctive path family $\{Q_i:i\in S\}$ in $D_0$ such that
			%\begin{enumerate}[label={\rm (\arabic*)}]
			%	\item for all $i\in S$, $Q_i$ is a (plus-minus) path of length at most $\theta$ that is incident to $a_i$,\label{item:short-paths-1}
			%	\item $\mathrm{dist}_{D_0^\diamond}(u,v)>k$ for all distinct $u,v\in \left(\bigcup_{i\in S} \{\ell(Q_i),r(Q_i)\}\right)\cup V(\mathcal{C})$ unless $\{u,v\}=\{\ell(Q_i),r(Q_i)\}$ for some $i\in S$.\label{item:short-paths-2}
			%\end{enumerate}		
			
			\no Consider the largest subset $S\subseteq [t]$ that admits a distinctive path family $\mathcal{Q}$ in $D_0$ satisfying \ref{item:Q-i-incident-a-i} and \ref{item:endpoints-far-away} (where $S$ plays the role of $[t]$) such that $|E(Q)|\leq \theta$ for all $Q\in \mathcal{Q}$. If $S=[t]$, we are done, so assume $[t]\setminus S$ is nonempty. Let $D'_0:=D_0-\mathcal{Q}$ and $F_0:=\bigcup_{Q\in \mathcal{Q}} \{\ell(Q),r(Q)\}$. For any $i\in [t]\setminus S$, we have $a_i\notin F_0$ using \ref{item:Q-i-incident-a-i}. Also, it is clear that $V^0(D'_0)\subseteq F_0$, so $\mathrm{SD}(D'_0,a_i)$ is well-defined for all $i\in [t]\setminus S$. Suppose for a contradiction that $\mathrm{SD}(D'_0,a_i)\leq \theta$ for some $i\in [t]\setminus S$. Then, we can choose a plus-minus path $Q_i$ in $D'_0$ of length $\mathrm{SD}(D'_0,a_i)$ that is incident to $a_i$. Let $x_i$ be the other endpoint of $Q_i$. Since $|E(Q)|\leq \theta$ for all $Q\in \mathcal{Q}$ and $|E(Q_i)|\leq \theta$, by \ref{item:DIRECTED-CYCLES-FAR-AWAY} and \eqref{eqn:inequality-theta}, we see that for all $x\in (F_0\cup V(\mathcal{C}))\setminus\{a_i\}$,	 
			\begin{align*}
				\mathrm{dist}_{D_0^\diamond}(x_i,x)\geq \log\log n-2\theta>(\log\log n)/3>k.
			\end{align*}
			In particular, this gives {$x_i\notin F_0\cup V(\mathcal{C})$}. By \ref{item:DIRECTED-CYCLES-DISJOINT}, we have $\mathcal{Q}\cup \{Q_i\}$ is a distinctive path family in $D_0$ satisfying \ref{item:Q-i-incident-a-i} and \ref{item:endpoints-far-away} (where $S\cup \{i\}$ plays the role of $[t]$) such that $|E(Q)|\leq \theta$ for all $Q\in \mathcal{Q}\cup \{Q_i\}$, which contradicts $S$ being the largest such subset of $[t]$. As a result, 
			\begin{align}
				\mathrm{SD}(D'_0,a_i)>\theta\text{ holds for all }i\in [t]\setminus S. \label{eqn:SD-theta-a-i}
			\end{align}

			\no Then, consider the largest subset $T\subseteq [t]$ that admits a distinctive path family $\mathcal{Q}'$ satisfying \ref{item:Q-i-incident-a-i} and \ref{item:endpoints-far-away} (where $T$ plays the role of $[t]$) such that $\mathcal{Q}\subseteq \mathcal{Q}'$. We will complete the proof of the claim by showing that $T=[t]$. Assume for a contradiction that there exists $i\in [t]\setminus T$. Let $D'_1:=D_0-\mathcal{Q}'$ and $F_1:=\bigcup_{Q\in \mathcal{Q}'} \{\ell(Q),r(Q)\}$. We have $a_i\notin F_1$ using \ref{item:Q-i-incident-a-i}. Also, it is clear that $V^0(D'_1)\subseteq F_1$, so $\mathrm{SD}(D'_1,a_i)$ is well-defined. Moreover, we have $\mathrm{SD}(D'_0,a_i)\leq \mathrm{SD}(D'_1,a_i)$, which, in particular, implies that $\mathrm{SD}(D'_1,a_i)> \theta$ using \eqref{eqn:SD-theta-a-i}. On the other hand, by \eqref{eqn:inequality-theta} and \ref{item:DIRECTED-N-SHORT-CYCLES}, we have
			\begin{align}
				|V^0(D'_1)|\leq |F_1|\leq 2t\leq 2\log\log\log n<\theta/3<\mathrm{SD}(D'_1,a_i)/3.\label{eqn:relating-theta-SD}
			\end{align}
			Hence, by using \eqref{eqn:relating-theta-SD} to apply \Cref{prop:PM(D)-relating-SD(v)-V0(D)}, we have $|\mathrm{PM}(D'_1,a_i)|\geq \log\log n$ using \eqref{eqn:choice-of-theta}. Moreover, by \ref{item:DIRECTED-N-SHORT-CYCLES} and \eqref{eqn:relating-theta-SD}, we obtain $|F_1\cup V(\mathcal{C})|\leq(p+2)\log\log\log n$. Note that, using $\Delta(G)\leq d$, {for any $x\in V(D_0)$ and $j\in[k]$, the number of vertices $w\in V(D_0)$ satisfying $\mathrm{dist}_{D_0^{\diamond}}(w,x)=j$ is at most $d(d-1)^{j-1}$. As $d(d-1)^{j-1}\leq d(d-1)^{k-1}$ for any $j\in [k]$,} the number of vertices $w\in V(D_0)$ satisfying $\mathrm{dist}_{D_0^{\diamond}}(w,x)\leq k$ for some $x\in F_1\cup V(\mathcal{C})$ is at most
			\begin{align*}
				(1+ {kd(d-1)^{k-1}})(p+2)\log\log\log n<\log \log n\leq |\mathrm{PM}(D'_1,a_i)|,
			\end{align*}
			as $n$ is sufficiently large. Hence, we can find $x_i\in \mathrm{PM}(D'_1,a_i)$ and a plus-minus path $Q_i$ in $D'_1$ with $\{\ell(Q_i),r(Q_i)\}=\{a_i,x_i\}$ such that $\mathrm{dist}_{D_0^{\diamond}}(x_i,x)> k$ for all $x\in F_1\cup V(\mathcal{C})$. This, in particular, implies that {$x_i\notin F_1\cup V(\mathcal{C})$}, so we have $\mathcal{Q}'\cup \{Q_i\}$ is a distinctive path family in $D_0$ satisfying \ref{item:Q-i-incident-a-i} and \ref{item:endpoints-far-away} (where $T\cup \{i\}$ playing the role of $[t]$), which contradicts $T$ being the largest such subset of $[t]$. We conclude that $T=[t]$, so the result follows.
		\end{proofclaim}

		\no By Claim~\ref{claim:incident-partial-paths}, let $\mathcal{Q}:=\{Q_i:i\in [t]\}$ be a distinctive path family in $D_0$ satisfying \ref{item:Q-i-incident-a-i} (with $f(i)=Q_i$ for $i\in [t]$) and \ref{item:endpoints-far-away}. For $i\in [t]$, let $x_i$ be the other endpoint of $Q_i$. By \ref{item:Q-i-incident-a-i}, note that $x_i\notin V(\mathcal{C})$. Recall that for every $i\in [t]$, either we have $a_i,b_i\in V^+(D)$ or $a_i,b_i\in V^-(D)$. Now, for each $i\in [t]$, we will modify $Q_i$ to obtain another plus-minus path $P_i$ in $D_0$.  If $a_i\in V^+(D)$, we take the last vertex $a'_i$ along $Q_i$ such that $a'_i\in V(C_i)$. If $a'_i\neq b_i$, we define $P_i=b_iC_ia'_iQ_ix_i$, otherwise we define $P_i=a_iC_ib_iQ_ix_i$. Similarly, if $a_i\in V^-(D)$, we take the first vertex $a'_i$ along $Q_i$ such that $a'_i\in V(C_i)$. If $a'_i\neq b_i$, we define $P_i=x_iQ_ia'_iC_ib_i$, otherwise we define $P_i=x_iQ_ib_iC_ia_i$. It is easy to check that for each $i\in [t]$, we have that $E(P_i)\cap E(C_i)\neq \emptyset$ and that $\{\ell(P_i),r(P_i)\}=\{x_i,y_i\}$ where $y_i\in \{a_i,b_i\}$. Hence, $\mathcal{P}:=\{P_i:i\in[t]\}$ is a distinctive path family in $D_0$ since $\{x_i,a_i,b_i:i\in [t]\}$ are distinct using that $\mathcal{Q}$ is a distinctive path family, $x_i\notin V(\mathcal{C})$, and~\ref{item:DIRECTED-CYCLES-DISJOINT}. Since $\mathrm{ex}_{D_0}(v)=\mathrm{ex}_D(v)\neq 0$ for all $v\in V(D)$, we can conclude that $\mathcal{P}$ is a partial path decomposition for $D$.\\
		
		\no Let $F=D-{\mathcal{P}}$. By \Cref{prop:PARTIAL-COMPLETION}, it suffices to prove that $F$ is consistent. Note that $E(P_i)\cap E(C_i)\neq \emptyset$ for all $i\in [t]$, so we have $g(F)>p\geq 1000d^2$. Also, it is clear that $\Delta^0(F)\leq d$. Hence, by \Cref{theorem:KEY}, it is enough to prove that $V^0(F)$ is $k$-sparse in $F$ as $k=20d^2+2$. Note that $E(F)\setminus E(D_0)\subseteq E(D)\setminus E(D_0)= E(\mathcal{C})$ and that $V^0(F)\subseteq \bigcup_{i\in [t]}\{x_i,y_i\}$ where $y_i\in \{a_i,b_i\}$ for all $i\in [t]$. Consider a path $P$ in $F^\diamond$ between some distinct $u,v\in V^0(F)$ with $|E(P)|\leq k$; we will prove that {$\{u,v\}=\{x_i,y_i\}$} for some $i\in [t]$ (which shows $V^0(F)$ is $k$-sparse in $F$). {Note that  $a_i\in V^0(F)$ (resp.~$b_i\in V^0(F)$) implies that $y_i=a_i$ (resp.~$y_i=b_i$) as we have $\mathrm{ex}_F(v) \not = 0$ for all $v\notin \bigcup_{j\in [t]}\{\ell(P_j),r(P_j)\}$. Moreover, $\{u,v\}=\{a_i,b_i\}$ is not possible as at most one of $a_i$ and $b_i$ belongs to $V^0(F)$. Hence, it suffices to prove that $u,v\in \{a_i,b_i,x_i\}$.} If $E(P)\cap E(\mathcal{C}^\diamond)=\emptyset$, then $\mathrm{dist}_{D_0^\diamond}(u,v)\leq k$. Hence, we have $\{u,v\}=\{a_i,x_i\}$ for some $i\in [t]$ using \ref{item:endpoints-far-away}, so we are done. Suppose $E(P)\cap E(\mathcal{C}^\diamond)\neq \emptyset$. {Let $z_1z'_1\in E(P)\cap E(\mathcal{C}^\diamond)$ be the first such edge encountered when traveling from $u$ to $v$ along $P$, with $z_1$ being the closest vertex to $u$ (possibly $z_1=u$ and $z'_1=v$). By~\ref{item:DIRECTED-CYCLES-DISJOINT}, there exists a unique $i\in [t]$ with $z_1,z'_1\in V(C_i)$. Note that $uPz_1\subseteq D_0^\diamond$ and $|E(uPz_1)|\leq k$. Then, by \ref{item:endpoints-far-away}, we have either $z_1=u$ or $\{z_1,u\}=\{a_{i'},x_{i'}\}$ for some $i'\in [t]$. Using $z_1\in V(C_i)$ and $u\in V^0(F)\subseteq \bigcup_{j\in [t]}\{a_j,b_j,x_j\}$, we can conclude that either $u=z_1\in \{a_i,b_i\}$ or  $i'=i$, $u=x_i$, $z_1=a_i$. This, in particular, implies that $u\in \{x_i,a_i,b_i\}$. Let $z_2\in V(P)$ be the closest vertex to $v$ along $z_1Pv$ satisfying $E(z_1Pz_2)\subseteq E(C_i^\diamond)$ (possibly $z_2=z'_1$ and/or $z_2=v$). It is clear that $z_2\in V(C_i)$. Moreover, by~\ref{item:DIRECTED-CYCLES-DISJOINT}, we have $z_2\notin V(C_j)$ for all $j\in [t]\setminus\{i\}$. Hence, if $v=z_2$, we obtain $v\in \{a_i,b_i\}$ since $v\in V^0(F)\subseteq \bigcup_{j\in [t]}\{a_j,b_j,x_j\}$, so we are done. Suppose that $v\neq z_2$.}\\
		
		\no Next, assume for a contradiction that $E(z_2Pv)\cap E(\mathcal{C}^\diamond)\neq \emptyset$. Let $z_3z'_3\in E(P)\cap E(\mathcal{C}^\diamond)$ be the first such edge encountered when traveling from $z_2$ to $v$ along $P$, with $z_3$ being closest vertex to $z_2$ (a priori possibly $z_3=z_2$ and $z'_3=v$). Note that $z_2Pz_3\subseteq D_0^\diamond$ and $|E(z_2Pz_3)|\leq k$.  Then, by \ref{item:endpoints-far-away}, we have either $z_2=z_3$ or $\{z_2,z_3\}=\{a_{i'},x_{i'}\}$ for some $i'\in [t]$. Using $z_2,z_3\in V(\mathcal{C})$ and $x_{i'}\notin V(\mathcal{C})$, the latter case is impossible, so we find $z_2=z_3$. However, using $z_2\notin V(C_j)$ for all $j\in [t]\setminus\{i\}$, this implies that $z_3z'_3\in E(C_i)$, which is a contradiction due to the definition of $z_2$. As a result, we have $E(z_2Pv)\cap E(\mathcal{C}^\diamond)=\emptyset$. Then, by \ref{item:endpoints-far-away}, we have $\{z_2,v\}=\{a_{i'},x_{i'}\}$ for some $i'\in [t]$. Again, using $v\in V^0(F)\subseteq \bigcup_{j\in [t]}\{a_j,b_j,x_j\}$ and $z_2\notin V(C_j)$ for all $j\in [t]\setminus\{i\}$, we obtain $i'=i$, $z_2=a_i$, $v=x_i$, so we are done. 
	\end{proof}
	
	\no Now, the proof of \Cref{thm:RANDOM-REGULAR} is immediate.
	
	\begin{proof}[Proof of \Cref{thm:RANDOM-REGULAR}]
		If $d=1$, then any orientation of $\mathcal{G}_{n,d}$ is just a matching, so the result trivially follows. Let $d\geq 3$ be an odd integer. By \Cref{prop:RANDOM-REGULAR}, $\mathcal{G}_{n,d}$ is $(n,d,1000d^2)$-discrete with high probability. Moreover, any orientation $D$ of $\mathcal{G}_{n,d}$ satisfies $\mathrm{ex}_D(v)=|d^+(v)-d^-(v)|\neq 0$ for all $v\in V(D)$ as $d^+(v)-d^-(v)$ has the same parity as $d$. Hence, the result follows by \Cref{thm:discrete-orientation}.    
	\end{proof}

	\section{Conclusion}\label{sec:conclusion}
	In this section, we point out possible directions for further research. Recall that  \Cref{theorem:LARGE-GIRTH} does not require the graph to be regular and so \Cref{corollary:main} holds not just for regular graphs but also for graphs $G$ with all degrees odd. One might wonder if Pullman's Conjecture (Conjecture~\ref{conjecture:Regular})  could also be true for graphs with all degrees odd. However, an example is given in \cite{CUBIC} of a graph with degrees $1$, $3$, or $5$ which has an orientation $D$ with $\mathrm{pn}(D)=4$ and $\mathrm{ex}(D)=3$. Here, we generalize that example by showing that the difference $\mathrm{pn}(D)-\mathrm{ex}(D)$ can be made arbitrarily large even when $\mathrm{ex}(v) \not= 0$ for all $v \in V(D)$.	
	
	%\begin{proposition}\label{prop:D0}
	%There exists an oriented graph $D_0$ on $5$ vertices with a special vertex $v$ such that $\mathrm{ex}(D_0)=2$, $\mathrm{pn}(D_0)=4$, $\mathrm{ex}(v)=0$, $d^T(v)=4$ and $\mathrm{ex}(w)=1$, $d^T(w)=3$ for all $w\neq v$.    
	%\end{proposition}
	\begin{proposition}\label{prop:large-degree}
		For each integer $k\geq 0$, there exists an oriented graph $G_k$ on $20k+10$ vertices with $\Delta(G_k^\diamond)=2k+5$ such that $d_{G_k^\diamond}(v)$ is odd for all $v\in V(G_k)$ and $\mathrm{pn}(G_k)-\mathrm{ex}(G_k)\geq 2k+2$. In particular, $G_k$ is not consistent.  
	\end{proposition}
	\begin{proof}
		First, consider the oriented graph $D_0$ with $\mathrm{ex}(D_0)=2$, depicted in Figure~\ref{fig:D0}.
		
		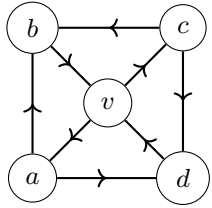
\begin{figure}[h]
			\centering
			
			\begin{tikzpicture}[scale=2]
				
				%\node at (0.5,1.2) {\large $D_0$};
				\node [style=whitecircle,scale=1.1] (1) at (0, 0) {\footnotesize$a$};
				\node [style=whitecircle,scale=1.1] (2) at (0, 1) {\footnotesize$b$};
				\node [style=whitecircle,scale=1.1] (3) at (1, 0) {\footnotesize$d$};
				\node [style=whitecircle,scale=1.1] (4) at (1, 1) {\footnotesize$c$};
				\node [style=whitecircle,scale=1.1] (a) at (0.5, 0.5) {\footnotesize$v$};

				\begin{scope}[thick,decoration={
						markings,
						mark=at position 0.5 with {\arrow{>}}}
					]
					\draw[postaction={decorate}] (1)--(2);
					\draw[postaction={decorate}] (4)--(2);
					\draw[postaction={decorate}] (1)--(3);
					\draw[postaction={decorate}] (4)--(3);
					\draw[postaction={decorate}] (a)--(1);
					\draw[postaction={decorate}] (a)--(4);
					\draw[postaction={decorate}] (2) to (a);
					\draw[postaction={decorate}] (3) to (a);
				\end{scope}

			\end{tikzpicture}
			\caption{The oriented graph $D_0$ on five vertices with $\mathrm{ex}(D_0)=2$ and $\mathrm{pn}(D_0)=4$.}
			\label{fig:D0}
		\end{figure}

		\no It can be easily checked that $\mathrm{pn}(D_0)=4$. Consider the oriented graph $G_k$ which consists of $4k+2$ copies of $D_0$ with additional 4k+1 edges, depicted in Figure~\ref{fig:large-degree-inconsistent}, where the unique excess-zero vertices in each copy of $D_0$ are shown as black circles. 
		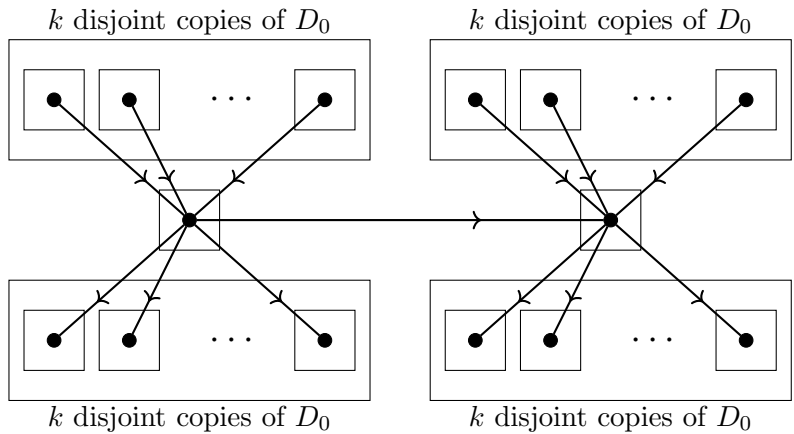
\begin{figure}[h]
			\centering
			
			\begin{tikzpicture}[scale=0.8]

				\node at (3.75,1) {\Large $\cdots$};
				\node at (10.75,1) {\Large $\cdots$};
				
				\node at (3,2.3) {$k$ disjoint copies of $D_0$};
				\draw  (0,0) -- (6,0) -- (6,2) -- (0,2) -- cycle ;

				\node at (10,2.3) {$k$ disjoint copies of $D_0$};
				\draw  (7,0) -- (13,0) -- (13,2) -- (7,2) -- cycle ;

				\draw  (0.25,1.5) -- (1.25,1.5) -- (1.25,0.5) -- (0.25,0.5) -- cycle ;
				\draw  (1.5,1.5) -- (2.5,1.5) -- (2.5,0.5) -- (1.5,0.5) -- cycle ;
				\draw  (4.75,1.5) -- (5.75,1.5) -- (5.75,0.5) -- (4.75,0.5) -- cycle ;
				
				\node [style=blackcircle,scale=0.5] (1) at (0.75, 1) {};
				\node [style=blackcircle,scale=0.5] (2) at (2,1) {};
				\node [style=blackcircle,scale=0.5] (3) at (5.25, 1) {};
				
				\draw  (7.25,1.5) -- (8.25,1.5) -- (8.25,0.5) -- (7.25,0.5) -- cycle ;
				\draw  (8.5,1.5) -- (9.5,1.5) -- (9.5,0.5) -- (8.5,0.5) -- cycle ;
				\draw  (11.75,1.5) -- (12.75,1.5) -- (12.75,0.5) -- (11.75,0.5) -- cycle ;
				
				\node [style=blackcircle,scale=0.5] (4) at (7.75, 1) {};
				\node [style=blackcircle,scale=0.5] (5) at (9,1) {};
				\node [style=blackcircle,scale=0.5] (6) at (12.25, 1) {};

				\node at (3.75,-3) {\Large $\cdots$};
				\node at (10.75,-3) {\Large $\cdots$};
				
				\node at (3,-4.3) {$k$ disjoint copies of $D_0$};
				\draw  (0,-4) -- (6,-4) -- (6,-2) -- (0,-2) -- cycle ;

				\node at (10,-4.3) {$k$ disjoint copies of $D_0$};
				\draw  (7,-4) -- (13,-4) -- (13,-2) -- (7,-2) -- cycle ;

				\draw  (0.25,-2.5) -- (1.25,-2.5) -- (1.25,-3.5) -- (0.25,-3.5) -- cycle ;
				\draw  (1.5,-2.5) -- (2.5,-2.5) -- (2.5,-3.5) -- (1.5,-3.5) -- cycle ;
				\draw  (4.75,-2.5) -- (5.75,-2.5) -- (5.75,-3.5) -- (4.75,-3.5) -- cycle ;
				
				\node [style=blackcircle,scale=0.5] (11) at (0.75, -3) {};
				\node [style=blackcircle,scale=0.5] (12) at (2,-3) {};
				\node [style=blackcircle,scale=0.5] (13) at (5.25, -3) {};
				
				\draw  (7.25,-2.5) -- (8.25,-2.5) -- (8.25,-3.5) -- (7.25,-3.5) -- cycle ;
				\draw  (8.5,-2.5) -- (9.5,-2.5) -- (9.5,-3.5) -- (8.5,-3.5) -- cycle ;
				\draw  (11.75,-2.5) -- (12.75,-2.5) -- (12.75,-3.5) -- (11.75,-3.5) -- cycle ;
				
				\node [style=blackcircle,scale=0.5] (14) at (7.75, -3) {};
				\node [style=blackcircle,scale=0.5] (15) at (9,-3) {};
				\node [style=blackcircle,scale=0.5] (16) at (12.25, -3) {};

				\draw  (2.5,-0.5) -- (3.5,-0.5) -- (3.5,-1.5) -- (2.5,-1.5) -- cycle ;
				
				\node [style=blackcircle,scale=0.5] (x) at (3, -1) {};
				
				\draw  (9.5,-0.5) -- (10.5,-0.5) -- (10.5,-1.5) -- (9.5,-1.5) -- cycle ;
				
				\node [style=blackcircle,scale=0.5] (y) at (10, -1) {};

				\begin{scope}[thick,decoration={
						markings,
						mark=at position 0.7 with {\arrow{>}}}
					]
					
					\draw[postaction={decorate}] (1)--(x);
					\draw[postaction={decorate}] (2)--(x);
					\draw[postaction={decorate}] (3)--(x);
					\draw[postaction={decorate}] (x)--(11);
					\draw[postaction={decorate}] (x)--(12);
					\draw[postaction={decorate}] (x)--(13);

					\draw[postaction={decorate}] (4)--(y);
					\draw[postaction={decorate}] (5)--(y);
					\draw[postaction={decorate}] (6)--(y);
					\draw[postaction={decorate}] (y)--(14);
					\draw[postaction={decorate}] (y)--(15);
					\draw[postaction={decorate}] (y)--(16);
					
					\draw[postaction={decorate}] (x)--(y);
				\end{scope}

				%\draw [style=THINBLACKARROW] (a) to (1a);
				%\path [style=THINBLACKARROW] (1) edge [bend left] node {} (4);
				%\path [style=THINBLACKARROW] (4) edge [bend left] node {} (1);
				
			\end{tikzpicture}
			\caption{The (inconsistent) oriented graph $G_k$ on $20k+10$ vertices with $\mathrm{ex}(v)=1$ for all $v\in V(G_k)$.}
			\label{fig:large-degree-inconsistent}
		\end{figure}
		Note that vertices in the middle copies have degree $2k+5$ in $G_k^\diamond$. Moreover, it is clear from the picture that all the vertices in $G_k$ {have} excess one, so $\mathrm{ex}(G_k)=10k+5$. On the other hand, since $\mathrm{pn}(D_0)=4$, we have $\mathrm{pn}(G_k)\geq 4(4k+2)-(4k+1)=12k+7$, which completes the proof.    
	\end{proof}
	
	\no We note that the girth condition in \Cref{theorem:LARGE-GIRTH} is unlikely to be optimal, although \Cref{prop:large-degree} shows that we cannot drop the girth condition altogether. It would be interesting to try and improve the girth condition in Theorem~\ref{theorem:LARGE-GIRTH} perhaps even down to a constant independent of the maximum degree. \\

	%\begin{conjecture}
	%Let $D$ be a digraph with $g(D)\geq \Delta^0(D)$ such that $\mathrm{ex}_D(v)\neq 0$ for all $v\in V(D)$. Then, %$D$ is consistent.     
	%\end{conjecture}

	\no 
	Of course Conjecture~\ref{conjecture:Regular} remains wide open. We have verified the conjecture in a sparse setting by imposing girth conditions on our graph. 
	%As mentioned above, it would be interesting to improve the girth condition down to a constant independent of degree. Unlike for \Cref{prop:large-degree}, there is no reason to have any girth condition for Conjecture~\ref{conjecture:Regular}, but imposing such a condition is a natural way to focus on the sparse setting. 
	Another direction for the sparse setting is to consider small values of the degree. 
	Recall that Conjecture~\ref{conjecture:Regular} is known to be true for $d=3$~\cite{CUBIC} using an inductive argument. However, the method does not seem to extend well to the case $d=5$, as the complexity of the structural analysis increases dramatically. A better understanding of the case $d=5$ could provide more insights about the sparse setting.  \\
	
	%\begin{conjecture}
	%	Every $5$-regular graph is strongly consistent.
	%\end{conjecture}

	\no 
	Conjecture~\ref{conjecture:Regular} is also wide open in the dense setting with the only solved case being that of the complete graph (as far as we are aware), and this required considerable effort relying heavily on the robust expansion method. It would be interesting to determine whether $K_{2n+1,2n+1}$ is strongly consistent.
	
	%While the special case of complete graphs required considerable effort, complete graphs may turn out to be one of the more difficult special cases of the conjecture. The reason is that one is forced to use very long paths (even Hamilton paths) when decomposing tournaments, which makes the problem for complete graphs quite delicate. There may be more room to manoeuvre for graphs of lower density. It would be interesting to determine to what extent this is the case by determining e.g.~whether $K_{2n+1,2n+1}$ is strongly consistent.

\end{document}